\documentclass[10pt]{amsart}

\usepackage{amsfonts}
\usepackage{mathrsfs}
\usepackage{mathtools}
\usepackage{amssymb}
\usepackage{latexsym,amsthm}
\usepackage[all]{xy}
\usepackage{graphicx}
\usepackage{wrapfig}
\usepackage{array,multirow}
\usepackage{bm}
\usepackage{enumitem,caption}
\usepackage{wasysym}

\allowdisplaybreaks[1]
\newcommand{\nc}{\newcommand}

\usepackage{hyperref}
\hypersetup{%
    pdfborder = {0 0 0},
    colorlinks,
    citecolor=red!50!black,
    filecolor=Darkgreen,
    linkcolor=blue!40!black,
	urlcolor=blue!50!cyan!50!black!90
}

\usepackage{tikz}
\usetikzlibrary{matrix,arrows,shapes,snakes}

\allowdisplaybreaks[1]

\newtheorem{thrm}{Theorem}[section]
\newtheorem{prop}[thrm]{Proposition}

\newtheorem{lemma}[thrm]{Lemma}

\theoremstyle{definition}
\newtheorem{conj}{Conjecture}
\newtheorem{defn}{Definition}
\newtheorem{exam}{Example}

\theoremstyle{remark}
\newtheorem{rmk}{Remark}

\newcommand{\rmkend}{\ensuremath{\diameter}}
\newcommand{\examend}{\ensuremath{\diameter}}
\newcommand{\defnend}{\ensuremath{\diameter}}

\nc{\ot}{\otimes}

\nc{\mc}{\mathcal}

\nc{\wt}{\widetilde}
\makeatletter
\newcommand*\rel@kern[1]{\kern#1\dimexpr\macc@kerna}
\newcommand*\wb[1]{%
  \begingroup
  \def\mathaccent##1##2{%
    \rel@kern{0.8}%
    \overline{\rel@kern{-0.8}\macc@nucleus\rel@kern{0.2}}%
    \rel@kern{-0.2}%
  }%
  \macc@depth\@ne
  \let\math@bgroup\@empty \let\math@egroup\macc@set@skewchar
  \mathsurround\z@ \frozen@everymath{\mathgroup\macc@group\relax}%
  \macc@set@skewchar\relax
  \let\mathaccentV\macc@nested@a
  \macc@nested@a\relax111{#1}%
  \endgroup
}
\makeatother

\nc{\al}{\alpha}
\nc{\ga}{\gamma}
\nc{\del}{\delta}
\nc{\la}{\lambda}
\nc{\om}{\omega}
\nc{\si}{\sigma}
\nc{\ze}{\zeta}

\nc{\Ga}{\Gamma}
\nc{\Del}{\Delta}
\nc{\Si}{\Sigma}

\nc{\Z}{\mathbb{Z}}
\nc{\R}{\mathbb{R}}
\nc{\C}{\mathbb{C}}
\nc{\K}{\mathbb{K}}

\nc{\mfg}{\mathfrak{g}}
\nc{\mfh}{\mathfrak{h}}
\nc{\mfk}{\mathfrak{k}}
\nc{\mfn}{\mathfrak{n}}
\nc{\mfz}{\mathfrak{z}}

\nc{\ad}{{\rm ad}}
\nc{\id}{\mathrm{id}}

\nc{\Ad}{{\rm Ad}}
\nc{\Aut}{{\rm Aut}}
\nc{\Autinv}{{\rm Aut}^{\rm inv}}
\nc{\CD}{\mathrm{CDec}}
\nc{\End}{\mathrm{End}}
\nc{\GL}{\mathrm{GL}}
\nc{\GSat}{\mathrm{GSat}}
\nc{\Hom}{\mathrm{Hom}}
\nc{\Id}{\mathrm{Id}}
\nc{\Sat}{\mathrm{Sat}}
\nc{\Sp}{\mathrm{Sp}}
\nc{\WSat}{\mathrm{WSat}}

\nc{\Idiff}{I_{\rm diff}}
\nc{\Ins}{I_{\rm ns}}
\nc{\Insf}{I_{\rm nsf}}

\nc{\mfgl}{\mathfrak{gl}}
\nc{\mfsl}{\mathfrak{sl}}
\nc{\mfso}{\mathfrak{so}}
\nc{\mfsp}{\mathfrak{sp}}

\nc{\qu}{\quad}
\nc{\qq}{\qquad}

\nc{\eq}[1]{\begin{equation}#1\end{equation}}
\nc{\eqa}[1]{\begin{equation}\begin{alignedat}{50}#1\end{alignedat}\end{equation}}
\nc{\eqn}[1]{\begin{equation*}\begin{alignedat}{50}#1\end{alignedat}\end{equation*}}
\nc{\eqg}[1]{\begin{equation}\begin{gathered}#1\end{gathered}\end{equation}}

\nc{\ali}[1]{\begin{alignat}{50}#1\end{alignat}}
\nc{\als}[1]{\begin{subequations}\begin{alignat}{50}#1\end{alignat}\end{subequations}}
\nc{\aln}[1]{\begin{alignat*}{50}#1\end{alignat*}}

\nc{\gat}[1]{\begin{gather}#1\end{gather}}
\nc{\gas}[1]{\begin{subequations}\begin{gather}#1\end{gather}\end{subequations}}
\nc{\gan}[1]{\begin{gather*}#1\end{gather*}}

\nc{\eqrefs}[2]{\text{(\ref{#1}-\ref{#2})}}

\numberwithin{equation}{section}

\nc{\red}{\color{red}}
\nc{\blu}{\color{blue}}
\nc{\br}{\color{brown}}
\nc{\grn}{\color{green!55!black}}
\nc{\gry}{\color{gray}}

\nc\nonweak{\begin{tikzpicture}[baseline=-0.25em,line width=0.7pt,scale=75/100]
\draw[double,->] (.9,0) -- (.95,0);
\draw[line width=.5pt] (.5,0) -- (1,0);
\draw[line width=.5pt] (.5,0.05) -- (1,0.05);
\draw[line width=.5pt] (.5,-0.05) -- (1,-0.05);
\filldraw[fill=black] (.5,0) circle (.1);
\filldraw[fill=white] (1,0) circle (.1);
\end{tikzpicture}}

\nc\weakmin{\begin{tikzpicture}[baseline=-0.25em,line width=0.7pt,scale=75/100]
\draw[double,<-] (4.05,0) --  (4.4,0);
\filldraw[fill=white] (4,0) circle (.1);
\filldraw[fill=black] (4.5,0) circle (.1);
\end{tikzpicture}}

\begin{document}

\title[Quasitriangular coideal subalgebras of $U_q(\mfg)$ in terms of generalized Satake diagrams]{Quasitriangular coideal subalgebras of $U_q(\mfg)$ \\ in terms of generalized Satake diagrams}

\author{Vidas Regelskis}
\address{Department of Mathematics, University of York, York, YO10 5DD, UK and \newline \mbox{\hspace{.35cm}} Institute of Theoretical Physics and Astronomy, Vilnius University, Saul\.etekio av.~3, Vilnius 10257, Lithuania.}
\email{vidas.regelskis@gmail.com}

\author{Bart Vlaar}
\address{
Department of Mathematics, University of York, York, YO10 5DD, UK and \newline \mbox{\hspace{.31cm}} Department of Mathematics, Heriot-Watt University, Edinburgh, EH14 4AS, UK}
\email{b.vlaar@hw.ac.uk}

\begin{abstract} 
Let $\mfg$ be a finite-dimensional semisimple complex Lie algebra and $\theta$ an involutive automorphism of $\mfg$.
According to G.~Letzter, S.~Kolb and M.~Balagovi\'c the fixed-point subalgebra $\mfk = \mfg^\theta$ has a quantum counterpart $B$, a coideal subalgebra of the Drinfeld-Jimbo quantum group $U_q(\mfg)$ possessing a universal K-matrix $\mc{K}$.
The objects $\theta$, $\mfk$, $B$ and $\mc{K}$ can all be described in terms of Satake diagrams.
In the present work we extend this construction to generalized Satake diagrams, combinatorial data first considered by A.~Heck.
A generalized Satake diagram naturally defines a semisimple automorphism $\theta$ of $\mfg$ restricting to the standard Cartan subalgebra $\mfh$ as an involution.
It also defines a subalgebra $\mfk\subset \mfg$ satisfying $\mfk \cap \mfh = \mfh^\theta$, but not necessarily a fixed-point subalgebra.
The subalgebra $\mfk$ can be quantized to a coideal subalgebra of $U_q(\mfg)$ endowed with a universal K-matrix in the sense of Kolb and Balagovi\'c.
We conjecture that all such coideal subalgebras of $U_q(\mfg)$ arise from generalized Satake diagrams in this way.
\end{abstract}


\maketitle

\vspace{-1em}
{
\parskip=0.1em
\setcounter{tocdepth}{1}
\tableofcontents
\setcounter{tocdepth}{2} % full TOC in PDF outline
}
\vspace{-3em}

%%%%%%%%%%%%%%%%%%%%%%%%%%%%%%%%%%%%%%%%%%%%%%%%%%%%%%%%%%%%%%%%%%%%%%%%%%%%%%
% Section 1: Introduction
%%%%%%%%%%%%%%%%%%%%%%%%%%%%%%%%%%%%%%%%%%%%%%%%%%%%%%%%%%%%%%%%%%%%%%%%%%%%%%

\section{Introduction}

Given a finite-dimensional semisimple complex Lie algebra $\mfg$ and an involutive Lie algebra automorphism $\theta \in \Aut(\mfg)$, a \emph{symmetric pair} is a pair $(\mfg,\mfk)$ where $\mfk = \mfg^\theta$ is the corresponding fixed-point subalgebra of $\mfg$, see \cite{Ar62,Sa71}.
\emph{Quantum symmetric pairs} are their quantum analogons.
That is, the enveloping algebra $U(\mfg)$ can be quantized to a quasitriangular Hopf algebra, the Drinfeld-Jimbo quantum group $U_q(\mfg)$ endowed with the universal R-matrix $\mc{R}$, see \cite{Ji85,Dr87}.
The quantum analogon of $\mfg^\theta$ is a coideal subalgebra $B \subseteq U_q(\mfg)$ \cite{Le99,Le02,Ko14} having a compatible quasitriangular structure, the universal K-matrix $\mc{K}$ \cite{BK19,Ko20} (see also \cite[Sec.~2.5]{BW18a} for the case of quantum symmetric pairs of type AIII/AIV).
Quantizations of symmetric pairs appeared earlier in a different approach in \cite{NDS95,NS95} (also see \cite{KS09}).
A prior notion of a universal K-matrix, not directly linked to a quantum symmetric pair, appeared in \cite{DKM03}.

The map $\theta$, the fixed-point subalgebra $\mfk$, the coideal subalgebra $B$ and the universal object $\mc{K}$ are all defined in terms of combinatorial information, the so-called Satake diagram $(X,\tau)$.
Here $X$ is a subdiagram of the Dynkin diagram of $\mfg$ and $\tau$ is an involutive diagram automorphism stabilizing $X$ and satisfying certain compatibility conditions, see \cite{Le02,Ko14}.

It is the aim of this paper to extend some of this work to a more general setting.
A direct motivation for this is the fact that the correct quantum group analogue of the fixed-point subalgebra in the Letzter-Kolb theory is not a fixed-point subalgebra itself, but merely tends to one as $q \to 1$, see \cite[Sec.~4]{Le99} and \cite[Ch.~10]{Ko14}.
This suggests that there may be a generalization of this theory that does not require a fixed-point subalgebra as input.

A careful analysis of \cite{Ko14,BK15,BK19} indeed indicates that the compatibility conditions for $X$ and $\tau$ can be weakened.
Indeed, in \cite[Rmks.~2.6 and 3.14]{BK15} it is explicitly suggested that some key passages of the theory are amenable for generalizations.
This leads to the notion of a \emph{generalized Satake diagram}, see Definition \ref{defn:GSat}, and the whole theory survives in this setting with minor adjustments.
The resulting Lie subalgebra $\mfk = \mfk(X,\tau)$ is given in Definition \ref{defn:k} and the corresponding coideal subalgebra $B = B(X,\tau)$ in Definition \ref{defn:B}.
For $\mfg$ of type A, all generalized Satake diagrams are Satake diagrams.
For other $\mfg$, the generalized Satake diagrams that are not Satake diagrams are listed in Table \ref{tab:nonSatake}.

Our proposed generalization of Satake diagrams can be traced back to the work of A.~Heck \cite{He84}.
These diagrams classify involutions of the root system of $\mfg$ such that the corresponding restricted Weyl group is the Weyl group of the restricted root system.
The characterization in terms of the restricted Weyl group is relevant in the context of the universal R- and K-matrices for quantum symmetric pairs.
The universal R-matrix $\mc{R}$ has a distinguished factor called quasi R-matrix playing an important role in the theory of canonical bases for $U_q(\mfg)$, see \cite{Ka90} and \cite[Part IV]{Lu94}.
The quasi R-matrix possesses a remarkable factorization property expressed in terms of the braid group action on $U_q(\mfg)$ of the Weyl group of~$\mfg$, see \cite{KR90,LS90}.
Recently it has become clear that many of these properties extend to the universal K-matrix~$\mc{K}$.
It has a distinguished factor called quasi K-matrix, introduced in \cite{BW18a} for certain coideal subalgebras of $U_q(\mfsl_N)$ and in a more general setting in \cite{BK15}.
This object plays a prominent role in the theory of canonical bases for quantum symmetric pairs \cite{BW18b}; for a historical note we refer the reader to \cite[Rmk.~4.9]{BW18b}.
In \cite{DK19} a factorization property is established for the quasi K-matrix using a braid group action of the restricted Weyl group.
As a consequence of the present work, this factorization property naturally extends to quasi K-matrices defined in terms of generalized Satake diagrams.

The Kac-Moody generalization of this approach will be addressed in a future work.
Another outstanding issue is a Lie-theoretic motivation of the subalgebra $\mfk$, which we define in an \emph{ad hoc} manner directly in terms of the combinatorial data $(X,\tau)$, see Definition \ref{defn:k}.
Therefore we now provide a further motivation for the study of the subalgebra $\mfk$ and its quantization $B$.

%%%%%%%%%%%%%%%%%%%%%%%%%%%%%%%%%%%%%%%%%%%%%%%%%%%%%%%%%%%%%%%%%%%%%%%%%%%%%%

\subsection{Some remarks on the representation theory of $(U_q(\mfg),B)$}

Consider the completion $\mc{U}$ of $U_q(\mfg)$ with respect to the category of integrable $U_q(\mfg)$-modules, so that objects in them have well-defined images under any finite-dimensional representation, see e.g.~\cite{Lu94,Jan96}.
Then $\mc{U} \ot \mc{U}$ can be embedded in a completion $\mc{U}^{(2)}$ of $U_q(\mfg)^{\ot 2}$ and one can construct an invertible $\mc R \in \mc{U}^{(2)}$ satisfying
\[
\mc{R} \Del(a) = \Del^{\rm op}(a) \mc{R} \text{ for all } a \in U_q(\mfg), \qq (\Del \ot \id)(\mc{R}) = R_{13}R_{23}, \qq (\id \ot \Del)(\mc{R}) = R_{13}R_{12},
\]
where $\Del$ is the coproduct and $\Del^{\rm op}$ the opposite coproduct (these can be viewed as maps from $\mc{U}$ to $\mc{U}^{(2)}$).
Analogously, according to \cite{BK19,Ko20}, one can construct an invertible $\mc{K} \in \mc{U}$ and an involutive Hopf algebra automorphism $\phi$ of $\mc{U}$ such that $(\phi \ot \phi)(\mc{R}) = \mc{R}$ and
\gat{
\label{Uintw} \mc{K} b = \phi(b) \mc{K} \qq \text{for all } b \in B, \\
\label{inBU} (\mc{R}^\phi)_{21} \mc{K}_2 \mc{R} \in \mc{B}^{(2)} , \\
\label{copr} \Del(\mc{K}) = \mc{R}_{21} (1 \ot \mc{K}) \mc{R}^\phi (\mc{K} \ot 1) ,
}
where $\mc{R}^\phi = (\phi \ot \id)(\mc{R})$, the subscript $_{21}$ denotes the simple transposition of tensor factors in $\mc{U}^{(2)}$ and $\mc{B}^{(2)} \subseteq \mc{U}^{(2)}$ is a particular completion of $B \ot U_q(\mfg)$, see \cite[Eq.~(3.31)]{Ko20}.
As a consequence, the \emph{(universal) $\phi$-twisted reflection equation} is satisfied:
\eq{ \label{URE}
\mc{R}_{21}\, (1 \ot \mc{K}) \, \mc{R}^\phi \, (\mc{K} \ot 1) = (\mc{K} \ot 1) \, (\mc{R}^\phi)_{21} \, (1 \ot \mc{K}) \,\mc{R} \qq \in \mc{U}^{(2)}.
}
The automorphism $\phi$ is given by $\tau \tau_0$ where $\tau_0$ is the diagram automorphism corresponding to the longest element of the Weyl group of $\mfg$.
The expression for $\mc{K}$ is given in \cite[Cor.~7.7]{BK19}.

One could argue in favour of making the automorphism $\phi$ inner: adjoin to $\mc{U}$ a group-like element $c_\phi$ such that $\phi(u) = c_\phi u c_\phi^{-1}$ for all $u \in \mc{U}$.
Then the object $\mc{K}_\phi := c_\phi^{-1} \mc{K}$ satisfies \eqrefs{Uintw}{copr} with $\phi$ replaced by $\id$.
However, for certain nontrivial diagram automorphisms $\phi$, $c_\phi$ cannot be chosen inside $\mc{U}$  so that $\mc{K}_\phi$ cannot be evaluated in all finite-dimensional representations.
This relates to the fact that the weights defining certain fundamental representations are not fixed by $\phi$.
For instance, if $\rho$ is the vector representation of $U_q(\mfsl_N)$ with $N>2$ one checks that the matrices $\rho(\phi(u))$ and $\rho(u)$ are not simultaneously similar for all $u \in U_q(\mfg)$.

Now let $\rho$ the vector representation of $U_q(\mfg)$; if $\mfg$ is of exceptional type by this we mean the smallest fundamental representation (for ${\rm E}_6$ one has a choice of two representations).
Choose $R \in \GL(V \ot V)$ proportional to $(\rho \ot \rho)(\mc{R})$, $R^\phi \in \GL(V \ot V)$ proportional to $(\rho \ot \rho)(\mc{R}^\phi)$ and $K \in \GL(V)$ proportional to $\rho(\mc{K})$.
Applying $\rho \ot \rho$ to \eqref{URE} one obtains the \emph{matrix reflection equation}
\eq{ \label{mRE}
R_{21}\, (\Id \ot K) \, R^\phi \, (K \ot \Id) = (K \ot \Id) \, (R^\phi)_{21}\, (\Id \ot K) \, R \qq \in \End(V\ot V)
}
where the subscript $_{21}$ indicates conjugation by the permutation operator in $\GL(V \ot V)$.
Starting with $\mfg$ of classical Lie type and a coideal subalgebra $B=B(X,\tau)$ where $(X,\tau)$ is a Satake diagram, the matrices $\rho(\mc{K})$ correspond to the solutions of \eqref{mRE} used in \cite{NDS95,NS95} to define quantum symmetric pairs.

Treating the matrix $R$ as given, one can of course solve \eqref{mRE} for $K \in \GL(V)$.
For $U_q(\mfsl_N)$ and $V=\C^N$ this was done by A.~Mudrov \cite{Mu02}.
From this result and computations for $U_q(\mfg)$ whose vector representation is of dimension at most 9 (i.e.~with $\mfg$ of types ${\rm B}_n$, ${\rm C}_n$, ${\rm D}_n$ ($n\le 4$) and ${\rm G}_2$) one obtains a classification of solutions $K$ of \eqref{mRE} for those pairs $(U_q(\mfg),\rho)$.
One can match this list of solutions $K$ one-to-one with a list of generalized Satake diagrams $(X,\tau)$ by checking which $K$ satisfies $K \rho(b) = \rho(\phi(b)) K$ for all $b \in B = B(X,\tau)$, i.e.~the image of \eqref{Uintw} under $\rho$.
Although this intertwining equation does not determine $K$ uniquely, it turns out that, provided $K\notin \C\,\Id$, each $K$ intertwines $\rho|_{B}$ for a unique $B = B(X,\tau)$ with $X$ not equal to the whole Dynkin diagram $I$.
In the case $X=I$ we must have $\tau = \tau_0$ and $B = U_q(\mfg)$ so that the excluded case $K \in \C \, \Id$ can be matched to it.
It leads to the following conjecture.

\begin{conj} \label{conj}
Let $\rho: U_q(\mfg) \to \End(V)$ be the vector representation of $U_q(\mfg)$.
\begin{enumerate}
\item
If $K \in \GL(V)$ is a solution of \eqref{mRE} there exists a generalized Satake diagram $(X,\tau)$ such that $K$ is proportional to $\rho(\mc{K})$ where $\mc{K} = \mc{K}(X,\tau)$ is the universal K-matrix for the subalgebra $B=B(X,\tau)$.\smallskip
\item
The only quasitriangular coideal subalgebras of $U_q(\mfg)$ are of the form $(B(X,\tau),\mc{K}(X,\tau))$ with $(X,\tau)$ a generalized Satake diagram.
\end{enumerate}
\end{conj}

In the Letzter-Kolb approach, the generators of the coideal subalgebra $B$ associated to a node $i \in I \backslash X$ carry extra parameters: scalars $\gamma_i \ne 0$ and $\si_i$, see Definition \ref{defn:B} and we can sharpen Conjecture \ref{conj} (i).
Namely, let $d_i$ denote the squared length of root $\al_i$ and write $q_i=q^{d_i}$.
Consider the set $I_{\rm ns} = \{ i \in I \backslash X \,|\, i \text{ does not neighbour } X, \, \tau(i)=i \}$, see~\eqref{specialorbits}, and the sets $\Ga_q$ and $\Si_q$, see~\eqref{parametersets:defn}; these definition go back to \cite{Le03,Ko14}.
Conjecturally, any invertible matrix solution $K$ of \eqref{mRE} is proportional to $\rho(\mc{K})$ for some $B(X,\tau)$ with $(X,\tau)$ a generalized Satake diagram whose parameters satisfy $(\ga_i)_{i \in I \backslash X} \in \Ga_q$, $\si_i = 0$ if $i \notin I_{\rm ns}$ and for all $(i,j) \in I_{\rm ns} \times I_{\rm ns}$ such that $i\ne j$ one of three conditions must hold: the Cartan integer $a_{ij}$ is even, $\si_j = 0$, or $(q_i-q_i^{-1})^2 \si_i^2 = - (q_i^{r/2}+q_i^{-r/2})^2 q_i \gamma_i$ for some odd positive $r \le -a_{ij}$.
The set $\Si_q$ does not cover the third possibility, which appeared in \cite{BB10} for $a_{ij} \in \{-1,-3\}$.
Conjecture \ref{conj} (ii) can be made more precise in an analogous way.

The approach in \cite{BK19} requires also certain constraints on $\ga_i$ and $\si_i$ under the transformation $q \to q^{-1}$ which are given in \eqref{gamma:bar} and \eqref{sigma:bar} in the present notation and generality.

%%%%%%%%%%%%%%%%%%%%%%%%%%%%%%%%%%%%%%%%%%%%%%%%%%%%%%%%%%%%%%%%%%%%%%%%%%%%%%

\subsection{Outline}

The paper is organized as follows.
In Section \ref{sec:Lie} we define the basic objects associated to a finite-dimensional semisimple complex Lie algebra $\mfg$ and its Cartan subalgebra~$\mfh$.
We introduce generalized Satake diagrams and explain how they emerge in the work of A.~Heck.

In Section \ref{sec:k} we define the Lie subalgebra $\mfk\subseteq\mfg$ in terms of $(X,\tau)$.
In Theorem \ref{mainthrm}, the main result of this section, we show that $\mfk$ satisfies $\mfk \cap \mfh = \mfh^\theta$ precisely if $(X,\tau)$ is a generalized Satake diagram.
In Propositions \ref{prop:kprime} and \ref{prop:k:weak:ideal} we describe the derived subalgebra of $\mfk$ and establish that when $\mfk$ is not reductive it is a semidirect product of a reductive subalgebra and a nilpotent ideal of class 2.
We end this section discussing the universal enveloping algebra $U(\mfk)$.

In Section \ref{sec:quantum} we indicate the necessary modifications to the papers \cite{Ko14,BK15,BK19,Ko20,DK19} so that they apply to the quantum pair algebras $B = U_q(\mfk)$ associated to generalized Satake diagrams.

\enlargethispage{1em}

Appendix \ref{appendix} contains three technical lemmas in aid of Section \ref{sec:k}.

We use the symbol $\defnend$ to indicate the end of definitions, examples and remarks.

%%%%%%%%%%%%%%%%%%%%%%%%%%%%%%%%%%%%%%%%%%%%%%%%%%%%%%%%%%%%%%%%%%%%%%%%%%%%%%
% Section 2
%%%%%%%%%%%%%%%%%%%%%%%%%%%%%%%%%%%%%%%%%%%%%%%%%%%%%%%%%%%%%%%%%%%%%%%%%%%%%%

\section{Finite-dimensional semisimple Lie algebras and root system involutions} \label{sec:Lie}

Let $I$ be a finite set and $A = (a_{ij})_{i,j \in I}$ a Cartan matrix.
In particular, there exist positive rationals $d_i$ ($i \in I$) such that $d_i a_{ij} = d_j a_{ji}$.
Let $\mfg = \mfg(A)$ be the corresponding finite-dimensional semisimple Lie algebra over $\C$.
It is generated by $\{ e_i, f_i,h_i \}_{i \in I}$ subject to
\gat{
\label{g:rel1} [h_i,h_j]=0, \qq [h_i,e_j] = a_{ij}e_j, \qq [h_i,f_j] = -a_{ij} f_j, \qq [e_i,f_j] = \del_{ij} h_i, \\
\label{g:Serre} \ad(e_i)^{M_{ij}}(e_j) = \ad(f_i)^{M_{ij}}(f_j) = 0 \qq \text{if } i \ne j,
}
for all $i,j \in I$, where we have set $M_{ij}:=1-a_{ij} \in \Z_{>0}$ if $i \ne j$.
The standard Cartan and nilpotent subalgebras are $\mfh = \langle h_i \,|\, i \in I \rangle$, $\mfn^+ = \langle e_i \,|\, i \in I \rangle$ and $\mfn^- = \langle f_i \,|\, i \in I \rangle$.

The simple roots $\al_i \in \mfh^*$ ($i \in I$) satisfy $\al_j(h_i) = a_{ij}$ for $i,j \in I$.
Let $Q = \sum_{i \in I} \Z \al_i$ denote the root lattice and write $Q^+ = \sum_{i \in I} \Z_{\ge 0} \al_i$.
For all $\al, \beta \in Q$, we write $\al > \beta$ if $\al-\beta \in Q^+ \backslash \{0\}$.
The Lie algebra $\mfg$ is $Q$-graded in terms of the root spaces $\mfg_\al = \{ x \in \mfg \, | \, [h,x] = \al(h)\, x \text{ for all } h \in \mfh \}$ and we have the following identities for $\mfh$-modules:
\eq{ \label{g:rootdecomp}
\mfg = \mfn^+ \oplus \mfh \oplus \mfn^-, \qq \mfn^\pm = \bigoplus_{\al \in Q^+} \mfg_{\pm \al}, \qq \mfh = \mfg_0.
}
Hence the root system $\Phi := \{ \al \in Q \, | \, \mfg_\al \ne \{0\}, \, \al \ne 0 \}$ satisfies $\Phi = \Phi^+ \cup \Phi^-$ where $\Phi^\pm = \pm(\Phi \cap Q^+)$.
The Weyl group $W$ is the (finite) subgroup of $\GL(\mfh^*)$ generated by the simple reflections $s_i$ ($i \in I$) acting via $s_i(\al) = \al - \al(h_i)\,\al_i$ for all $i \in I$, $\al \in \mfh^*$.
We define
\begin{align}
\Aut(\Phi) &= \{ g \in \GL(\mfh^*) \, | \, g(\Phi) = \Phi \}, \\
\Aut(A) &= \{ \si : I \to I \text{ invertible} \,\, | \,\, a_{\si(i) \, \si(j)} = a_{ij} \text{ for all } i,j \in I \}.
\end{align}
Then $\Aut(\Phi) = W \rtimes \Aut(A)$, with $\Aut(A)$ acting by relabelling.

We briefly review some important subgroups of
\eq{
\Aut(\mfg,\mfh) = \{ \si \in \Aut(\mfg) \, | \, \si(\mfh) = \mfh \} < \Aut(\mfg).
}
We have $\Aut(A) < \Aut(\mfg,\mfh)$ (acting by relabelling).
Also, a braid group action on $\mfg$ is given by $\Ad(s_i) = \exp(\ad(e_i)) \exp(\ad(-f_i)) \exp(\ad(e_i)) \in \Aut(\mfg,\mfh)$ for $i \in I$.
It extends the action of $W$ on $\mfh$ dual to the one on $\mfh^*$ and satisfies $\Ad(W) < \Aut(\mfg,\mfh)$.
The Chevalley involution $\om \in \Aut(\mfg,\mfh)$ is defined by swapping $e_i$ and $-f_i$ for all $i \in I$; it commutes with $\Ad(W)$ and with $\Aut(A)$.
Finally, the group $\wt H := \Hom(Q,\C^\times)$ naturally induces a subgroup $\Ad(\wt H) < \Aut(\mfg,\mfh)$ via $\Ad(\chi)|_{\mfg_\al} = \chi(\al)\, \id_{\mfg_\al}$ for all $\chi \in \wt H$, $\al \in Q$.

The elements of $\Aut(\mfg,\mfh)$ can be dualized to elements of $\Aut(\Phi)$.
Conversely, since $-\id_{\mfh^*} \in \Aut(\Phi)$ and $\Aut(\Phi) = W \rtimes \Aut(A)$, given $g \in \Aut(\Phi)$ there exists a unique $(w,\tau) \in W \times \Aut(A)$ such that $g = - w \tau$.
Then $\psi = \Ad(w)\, \omega \tau \in \Aut(\mfg,\mfh)$ satisfies $(\psi|_{\mfh})^* = g$.

%%%%%%%%%%%%%%%%%%%%%%%%%%%%%%%%%%%%%%%%%%%%%%%%%%%%%%%%%%%%%

\subsection{Compatible decorations and involutions of $\Phi$}

Given a subset $X \subseteq I$ denote the corresponding Cartan submatrix by $A_X = (a_{ij})_{i,j \in X}$ and consider the semisimple Lie algebra $\mfg_X := \langle e_i,f_i,h_i \, | \, i \in X \rangle \subseteq \mfg$ with Cartan subalgebra $\mfh_X= \mfh \cap \mfg_X$ and dual Weyl vector $\rho^\vee_X \in \mfh_X$.
The unique longest element $w_X$ of the Weyl group $W_X := \langle s_i \, | \, i \in X \rangle$ is an involution and there exists $\tau_{0,X} \in \Aut(A_X)$ which satisfies 
\eq{ \label{wX:node}
-w_X(\al_i) = \al_{\tau_{0,X}(i)} \qq \text{for all } i \in X.
}
We recall here the basic fact that both $w_X$ and $\tau_{0,X}$ naturally factorize with respect to the decomposition of $X$ into connected components.
Furthermore, if $X$ is connected then $\tau_{0,X}$ is trivial unless $X$ is of type ${\rm A}_n$ with $n>1$, ${\rm D}_n$ with $n>4$ odd or ${\rm E}_6$ (in each case of which there is a unique nontrivial diagram automorphism).
Note that $\Ad(w_X)|_{\mfg_X} =  \tau_{0,X}\,\om|_{\mfg_X}$ and $\Ad(w_X)^2 = \Ad(\zeta)$, where $\zeta \in \wt H$ is defined by $\zeta(\al) = (-1)^{\al(2\rho^\vee_X)}$ for all $\al \in Q$.

We can describe
\ali{
\Autinv(\mfg,\mfh) &:= \{ \psi \in \Aut(\mfg,\mfh) \,\, | \,\, \psi^2|_{\mfh} = \id_{\mfh} \} , \\
\Autinv(\Phi) &:= \{ g \in \Aut(\Phi) \,\, | \,\, g^2 = \id_{\mfh^*} \}
}
by combinatorial data.
Define the set of \emph{compatible decorations} as
\eq{
\CD(A) = \{ (X,\tau) \,\, | \,\, X \subseteq I, \, \tau \in \Aut(A), \, \tau^2 = \id_I,  \, \tau(X) = X, \, \tau|_X = \tau_{0,X} \}.
}
In the associated Dynkin diagram one marks a compatible decoration by filling the nodes corresponding to $X$ and drawing bidirectional arrows for the nontrivial orbits of~$\tau$.

\begin{exam} \label{ex1} 
Let $A$ be of type ${\rm A}_n$, $n\ge 2$.
The compatible decorations are 
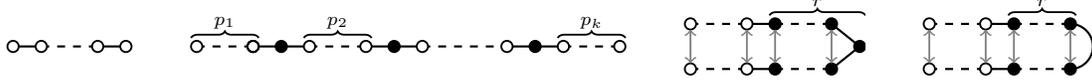
\[
\begin{tikzpicture}[baseline=-2em,line width=0.7pt,scale=75/100]
\clip (-.7,-.4) rectangle (1.7,.4);
\draw[thick] (-.5,0) -- (0,0);
\draw[thick,dashed] (0,0) -- (1,0);
\draw[thick] (1,0) --  (1.5,0);
\filldraw[fill=white] (-.5,0) circle (.1);
\filldraw[fill=white] (0,0) circle (.1);
\filldraw[fill=white] (1,0) circle (.1);
\filldraw[fill=white] (1.5,0) circle (.1);
\end{tikzpicture}
\qq
\begin{tikzpicture}[baseline=-2em,line width=0.7pt,scale=75/100]
\draw[thick,dashed] (0,0) -- (1,0);
\draw[thick] (1,0) --  (2,0);
\draw[thick,dashed] (2,0) -- (3,0);
\draw[thick] (3,0) --  (4,0);
\draw[thick,dashed] (4,0) -- (5.5,0);
\draw[thick] (5.5,0) -- (6.5,0);
\draw[thick,dashed] (6.5,0) -- (7.5,0);
\draw[snake=brace] 
(-.1,.15) -- (1.1,.15) node[midway,above]{\scriptsize $p_1$}
(1.9,.15) -- (3.1,.15) node[midway,above]{\scriptsize $p_2$}
(6.4,.15) -- (7.6,.15) node[midway,above]{\scriptsize $p_k$};
\filldraw[fill=white] 
(0,0) circle (.1) (1,0) circle (.1) (2,0) circle (.1) (3,0) circle (.1) (4,0) circle (.1) (1,0) circle (.1) (5.5,0) circle (.1) (6.5,0) circle (.1) (7.5,0) circle (.1);
\filldraw[fill=black] (1.5,0) circle (.1) (3.5,0) circle (.1) (6,0) circle (.1);
\end{tikzpicture}
\qq
\begin{tikzpicture}[baseline=-2em,line width=0.7pt,scale=75/100]
\clip (1.3,-.6) rectangle (4.6,1.1);
\draw[thick,dashed] (1.5,.4) -- (2.5,.4);
\draw[thick] (2.5,.4) -- (3,.4);
\draw[thick,dashed] (3,.4) -- (4,.4);
\draw[thick] (4,.4) -- (4.5,0) -- (4,-.4);
\draw[thick,dashed] (1.5,-.4) -- (2.5,-.4);
\draw[thick] (2.5,-.4) -- (3,-.4);
\draw[thick,dashed] (3,-.4) -- (4,-.4);
\draw[snake=brace] (2.9,.55) -- (4.6,.55) node[midway,above]{\scriptsize $r$};
\filldraw[fill=white] (1.5,.4) circle (.1);
\filldraw[fill=white] (1.5,-.4) circle (.1);
\draw[<->,gray] (1.5,.3) -- (1.5,-.3);
\filldraw[fill=white] (2.5,.4) circle (.1);
\filldraw[fill=white] (2.5,-.4) circle (.1);
\draw[<->,gray] (2.5,.3) -- (2.5,-.3);
\filldraw[fill=black] (3,.4) circle (.1);
\filldraw[fill=black] (3,-.4) circle (.1);
\draw[<->,gray] (3,.3) -- (3,-.3);
\filldraw[fill=black] (4,.4) circle (.1);
\filldraw[fill=black] (4,-.4) circle (.1);
\draw[<->,gray] (4,.3) -- (4,-.3);
\filldraw[fill=black] (4.5,0) circle (.1);
\end{tikzpicture} 
\qq
\begin{tikzpicture}[baseline=-2em,line width=0.7pt,scale=75/100]
\clip (1.3,-.6) rectangle (4.6,1.1);
\draw[thick,dashed] (1.5,.4) -- (2.5,.4);
\draw[thick] (2.5,.4) -- (3,.4);
\draw[thick,dashed] (3,.4) -- (4,.4);
\draw[thick,domain=270:450] plot({4+.4*cos(\x)},{.4*sin(\x)});
\draw[thick,dashed] (1.5,-.4) -- (2.5,-.4);
\draw[thick] (2.5,-.4) -- (3,-.4);
\draw[thick,dashed] (3,-.4) -- (4,-.4);
\draw[snake=brace] (2.9,.55) -- (4.1,.55) node[midway,above]{\scriptsize $r$};
\filldraw[fill=white] (1.5,.4) circle (.1);
\filldraw[fill=white] (1.5,-.4) circle (.1);
\draw[<->,gray] (1.5,.3) -- (1.5,-.3);
\filldraw[fill=white] (2.5,.4) circle (.1);
\filldraw[fill=white] (2.5,-.4) circle (.1);
\draw[<->,gray] (2.5,.3) -- (2.5,-.3);
\filldraw[fill=black] (3,.4) circle (.1);
\filldraw[fill=black] (3,-.4) circle (.1);
\draw[<->,gray] (3,.3) -- (3,-.3);
\filldraw[fill=black] (4,.4) circle (.1);
\filldraw[fill=black] (4,-.4) circle (.1);
\draw[<->,gray] (4,.3) -- (4,-.3);
\end{tikzpicture}
\]
where $k \in \Z_{\ge 2}$, $p_1,p_k\in\Z_{\ge0}$, $p_2,\dots,p_{k-1}\in\Z_{\ge1}$ and $r$, the number of $\tau$-orbits in $X$, is constrained by $0\le r \le \lceil n/2 \rceil$.
\hfill \examend
\end{exam}

Given $(X,\tau) \in \CD(A)$, we define 
\eq{ \label{theta:defn}
\theta = \theta(X,\tau) = - w_X \tau \in \Autinv(\Phi).
}
As explained above, the map dual to $\theta$ can be extended to an element of $\Autinv(\mfg,\mfh)$, also called $\theta$ and given by $\theta =\Ad(w_X)\,\tau \om$.
Owing to aforementioned properties of $\Ad(w_X)$ we have
\eq{ 
\label{theta:basic} \theta|_{\mfg_X} = \id_{\mfg_X}, \qq \theta^2 = \Ad(\zeta).
}
Note that $(\id-\theta)(h_i-h_{\tau(i)})$ lies in $\mfh_X \subseteq \mfh^\theta$ for all $i \in I$.
Hence it vanishes so that
\eq{
\label{theta:h} \theta(h_i-h_{\tau(i)}) = h_i-h_{\tau(i)}.
}
We fix a subset $I^* \subseteq I \backslash X$ containing precisely one element from each $\tau$-orbit in $I\backslash X$.
As a consequence of \eqref{theta:h} we have
\eq{
\label{htheta:expl}
\mfh^\theta = \bigoplus_{i \in X} \C h_i \oplus \bigoplus_{\substack{i \in I^* \\ i \ne \tau(i)}} \C(h_i-h_{\tau(i)}).
}

%%%%%%%%%%%%%%%%%%%%%%%%%%%%%%%%%%%%%%%%%%%%%%%%%%%%%%%%%%%%%

\subsection{Generalized Satake diagrams and the restricted Weyl group}

For $i \in I \backslash X$ denote by $\check X(i)$ the union of connected components of $X$ neighbouring $\{ i, \tau(i) \}$.

\begin{defn} \label{defn:GSat}
\emph{Generalized Satake diagrams} are elements of the set 
\begin{flalign*}
&& \GSat(A) := \Big\{ (X,\tau) \in \CD(A) \,\, \Big| \,\, \forall i \in I \backslash X: \, \check X(i) \cup \{ i, \tau(i)\} \ne \begin{tikzpicture}[baseline=-0.3em,line width=0.7pt,scale=75/100]
\draw[thick] (.5,0) -- (1,0);
\filldraw[fill=white] (.5,0) circle (.1);
\filldraw[fill=black] (1,0) circle (.1);
\end{tikzpicture}
\Big\}.
&& \defnend
\end{flalign*}
\end{defn}

The compatible decorations considered in Example \ref{ex1} lie in $\GSat(A)$ if and only if $p_1=p_k=0$ and $p_2=\ldots=p_{k-1}=1$.

\begin{rmk}
Generalized Satake diagrams were first considered by A.~Heck in \cite{He84}.
He uses the term ``Satake diagrams'' in a more general setting, see \cite[\S 1 - \S 2]{He84}: he starts with $\si = -\theta \in \Aut^{\rm inv}(\Phi)$ and calls a base $\Pi$ of $\Phi$ $\si$-fundamental if for all $\al \in \Pi$ either $\theta(\al)=\al$ or $\theta(\al) \in \Z_{\le 0} \Pi$.
Letting $X$ consist of the nodes corresponding to $\Pi^\theta$ in the Dynkin diagram corresponding to $\Pi$, it follows that $\tau := \si w_X$ is an involutive diagram automorphism restricting on $X$ to $\tau_{0,X}$.
He calls $(X,\tau)$ the Satake diagram of $\si$, which we call a compatible decoration; what he calls an ``admissible Satake diagram'' is in our case a generalized Satake diagram.
Since the term ``Satake diagram'' has come to be associated to involutions of the complex Lie algebra $\mfg$, we prefer the nomenclature ``compatible decoration'' and ``generalized Satake diagram''.
\hfill \rmkend
\end{rmk}

Note that $(X,\tau) \in \CD(A)$ is a generalized Satake diagram precisely if 
\eq{ \label{GSat:rephrased}
 \forall (i,j) \in I\backslash X \times X \,: \, \tau(i)=i, \, w_X(\al_i) = \al_i+\al_j \, \implies \, a_{ij} \ne -1 ,
}
which is the condition needed in \cite[Proof of Lemma 5.11, Step 1]{Ko14} and \cite[Proof of Lemma 6.4]{BK19}.
Straightforwardly one checks that it is equivalent to any of the following conditions:
\begin{gather}
\forall (i,j) \in I\backslash X \times X \,: \, \tau(i)=i, \, X(i) = \{ j \} \, \implies a_{ij}a_{ji} \ne 1, \\
\forall i,j \in I: \, \theta(\al_i) = -(\al_i+\al_j) \, \implies \, a_{ij} \ne -1, \\
\forall i \in I: \, (\theta(\al_i))(h_i) \ne -1.
\end{gather}

Satake diagrams can be defined as the following subset of compatible decorations of $A$:
\eq{ \label{Sat:defn}
\Sat(A) = \{ (X,\tau) \in \CD(A) \,\, | \,\, \forall i \in I \backslash X:\, i = \tau(i) \, \implies \, \al_i(\rho^\vee_X) \in \Z \}.
}
Satake diagrams classify involutive Lie algebra automorphisms up to conjugacy, see e.g.~\cite{Ar62}.
In our notation, for $(X,\tau) \in \GSat(A)$ and $\bm \ga \in (\C^\times)^{I^*}$ define $\chi_{\bm \ga} \in \wt H$ and $\theta_{\bm \ga} \in \Aut(\mfg)$ by 
\eq{ \label{theta_gamma}
\chi_{\bm \ga}(\al_i) = \begin{cases} 
1 & \text{if } i \in X , \\
\ga_i & \text{if } i \in I^*, \\
\ga_{\tau(i)} \zeta(\al_i) & \text{if } i \in I \backslash (X \cup \backslash I^*),
\end{cases} 
\qq \qq \qq
\theta_{\bm \ga} = \Ad(\chi_{\bm \ga})\,\theta
}
and note that \eqref{theta:basic} implies $\theta_{\bm \ga}^2 = \id_{\mfg}$ if $(X,\tau) \in \Sat(A)$.

We have $\Sat(A) \subseteq \GSat(A)$; indeed, if $(X,\tau) \in \CD(A) \backslash \GSat(A)$ there exists $(i,j) \in I \backslash X \times X$ such that $\tau(i)=i$, $\check X(i) = \{j\}$ and $a_{ij}=a_{ji}=-1$, so that $\al_i(\rho^\vee_X) = a_{ji}/2 \notin \Z$ implying $(X,\tau) \notin \Sat(A)$.
We refer the reader to the classification of generalized Satake diagrams in \cite[Table I]{He84}.
Since this does not distinguish between elements of $\Sat(A)$ and $\GSat(A) \backslash \Sat(A)$, for convenience we list the elements of $\GSat(A) \backslash \Sat(A)$, see Table \ref{tab:nonSatake}; note that outside type ${\rm A}_n$ we have $\GSat(A) \ne \Sat(A)$.

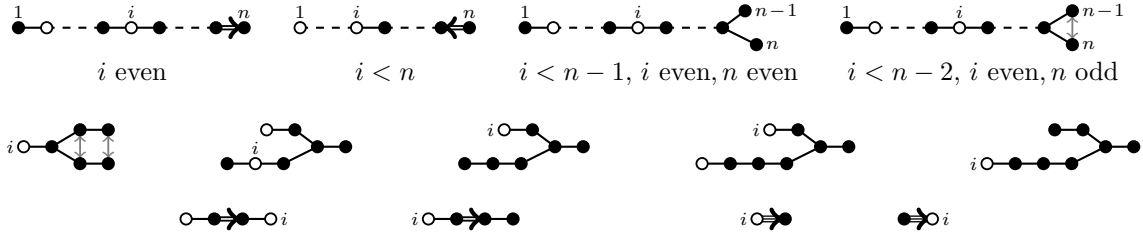
\begin{table}[h] 
\caption{The set $\GSat(A) \backslash \Sat(A)$ for indecomposable Cartan matrices $A$.
By a case-by-case analysis there is a unique $i \in I$ such that $i = \tau(i)$ and $\al_i(\rho^\vee_X) \notin \Z$ and we have indicated that node in the diagrams.
The classical families of diagrams are labelled in the standard way.
For types ${\rm C}_n$ and ${\rm D}_n$ upper bounds on $i$ are imposed to avoid the cases when $\theta$ is an involution whose fixed-point subalgebra is isomorphic to $\mfgl_n$.
}
\label{tab:nonSatake}

\vspace{-1em}

\begin{gather*}
\begin{array}{cccc}
\begin{tikzpicture}[baseline=-0.25em,line width=0.7pt,scale=75/100]
\draw[thick] (1,0) -- (1.5,0);
\draw[thick,dashed] (1.5,0) -- (2.5,0);
\draw[thick] (2.5,0) -- (3.5,0);
\draw[thick,dashed] (3.5,0) -- (4.5,0);
\draw[double,->] (4.6,0) --  (4.95,0);
\filldraw[fill=black] (1,0) circle (.1) node[above]{\scriptsize $1$};
\filldraw[fill=white] (1.5,0) circle (.1);
\filldraw[fill=black] (2.5,0) circle (.1);
\filldraw[fill=white] (3,0) circle (.1) node[above]{\scriptsize $i$};
\filldraw[fill=black] (3.5,0) circle (.1);
\filldraw[fill=black] (4.5,0) circle (.1);
\filldraw[fill=black] (5,0) circle (.1) node[above]{\scriptsize $n$};
\end{tikzpicture} 
&
\begin{tikzpicture}[baseline=-0.25em,line width=0.7pt,scale=75/100]
\draw[thick,dashed] (1.5,0) -- (2.5,0);
\draw[thick] (2.5,0) -- (3,0);
\draw[thick,dashed] (3,0) -- (4,0);
\draw[double,<-] (4.05,0) --  (4.4,0);
\filldraw[fill=white] (1.5,0) circle (.1) node[above]{\scriptsize $1$};
\filldraw[fill=white] (2.5,0) circle (.1) node[above]{\scriptsize $i$};
\filldraw[fill=black] (3,0) circle (.1);
\filldraw[fill=black] (4,0) circle (.1);
\filldraw[fill=black] (4.5,0) circle (.1) node[above]{\scriptsize $n$};
\end{tikzpicture} 
&
\begin{tikzpicture}[baseline=-0.25em,line width=0.7pt,scale=75/100]
\draw[thick] (1,0) -- (1.5,0);
\draw[thick,dashed] (1.5,0) -- (2.5,0);
\draw[thick] (2.5,0) -- (3.5,0);
\draw[thick,dashed] (3.5,0) -- (4.5,0);
\draw[thick] (4.9,.3) -- (4.5,0) -- (5.1,-.3);
\filldraw[fill=black] (1,0) circle (.1) node[above]{\scriptsize $1$};
\filldraw[fill=white] (1.5,0) circle (.1);
\filldraw[fill=black] (2.5,0) circle (.1);
\filldraw[fill=white] (3,0) circle (.1) node[above]{\scriptsize $i$};
\filldraw[fill=black] (3.5,0) circle (.1);
\filldraw[fill=black] (4.5,0) circle (.1);
\filldraw[fill=black] (4.9,.3) circle (.1) node[right]{\scriptsize $n\!-\!1$};
\filldraw[fill=black] (5.1,-.3) circle (.1) node[right]{\scriptsize $n$};
\end{tikzpicture}
&
\begin{tikzpicture}[baseline=-0.25em,line width=0.7pt,scale=75/100]
\draw[thick] (1,0) -- (1.5,0);
\draw[thick,dashed] (1.5,0) -- (2.5,0);
\draw[thick] (2.5,0) -- (3.5,0);
\draw[thick,dashed] (3.5,0) -- (4.5,0);
\draw[thick] (5,.3) -- (4.5,0) -- (5,-.3);
\draw[<->,gray] (5,.2) -- (5,-.2);
\filldraw[fill=black] (1,0) circle (.1) node[above]{\scriptsize $1$};
\filldraw[fill=white] (1.5,0) circle (.1);
\filldraw[fill=black] (2.5,0) circle (.1);
\filldraw[fill=white] (3,0) circle (.1) node[above]{\scriptsize $i$};
\filldraw[fill=black] (3.5,0) circle (.1);
\filldraw[fill=black] (4.5,0) circle (.1);
\filldraw[fill=black] (5,.3) circle (.1) node[right]{\scriptsize $n\!-\!1$};
\filldraw[fill=black] (5,-.3) circle (.1) node[right]{\scriptsize $n$};
\end{tikzpicture}
\\
i \text{ even} & i<n & i<n-1, \, i \text{ even}, n \text{ even} & i<n-2, \, i \text{ even}, n \text{ odd}
\end{array}
\\[1mm]
\begin{tikzpicture}[baseline=-0.35em,line width=0.7pt,scale=75/100]
\draw[thick] (1,.3) -- (.5,.3) -- (0,0) -- (.5,-.3) -- (1,-.3);
\draw[thick] (0,0) -- (-.5,0);
\draw[<->,gray] (1,.2) -- (1,-.2);
\draw[<->,gray] (.5,.2) -- (.5,-.2);
\filldraw[fill=black] (1,.3) circle (.1);
\filldraw[fill=black] (1,-.3) circle (.1);
\filldraw[fill=black] (.5,.3) circle (.1);
\filldraw[fill=black] (.5,-.3) circle (.1);
\filldraw[fill=black] (0,0) circle (.1);
\filldraw[fill=white] (-.5,0) circle (.1) node[left]{\scriptsize $i$};
\end{tikzpicture}
\qq \qq
\begin{tikzpicture}[baseline=-0.35em,line width=0.7pt,scale=75/100]
\draw[thick] (-.9,.3) -- (-.4,.3) -- (0,0) -- (-.6,-.3) -- (-1.6,-.3);
\draw[thick] (0,0) -- (.5,0);
\filldraw[fill=white] (-.9,.3) circle (.1);
\filldraw[fill=black] (-.4,.3) circle (.1);
\filldraw[fill=black] (0,0) circle (.1);
\filldraw[fill=black] (-.6,-.3) circle (.1);
\filldraw[fill=white] (-1.1,-.3) circle (.1) node[above]{\scriptsize$i$};
\filldraw[fill=black] (-1.6,-.3) circle (.1);
\filldraw[fill=black] (.5,0) circle (.1);
\end{tikzpicture}
\qq \qq
\begin{tikzpicture}[baseline=-0.35em,line width=0.7pt,scale=75/100]
\draw[thick] (-.9,.3) -- (-.4,.3) -- (0,0) -- (-.6,-.3) -- (-1.6,-.3);
\draw[thick] (0,0) -- (.5,0);
\filldraw[fill=white] (-.9,.3) circle (.1) node[left]{\scriptsize $i$};
\filldraw[fill=black] (-.4,.3) circle (.1);
\filldraw[fill=black] (0,0) circle (.1);
\filldraw[fill=black] (-.6,-.3) circle (.1);
\filldraw[fill=black] (-1.1,-.3) circle (.1);
\filldraw[fill=black] (-1.6,-.3) circle (.1);
\filldraw[fill=black] (.5,0) circle (.1);
\end{tikzpicture}
\qq \qq
\begin{tikzpicture}[baseline=-0.35em,line width=0.7pt,scale=75/100]
\draw[thick] (-2.1,-.3) -- (-.6,-.3) -- (0,0) -- (-.4,.3) -- (-.9,.3);
\draw[thick] (0,0) -- (.5,0);
\filldraw[fill=white] (-.9,.3) circle (.1) node[left]{\scriptsize $i$};
\filldraw[fill=black] (.5,0) circle (.1);
\filldraw[fill=black] (-.4,.3) circle (.1);
\filldraw[fill=black] (0,0) circle (.1);
\filldraw[fill=black] (-.6,-.3) circle (.1);
\filldraw[fill=black] (-1.1,-.3) circle (.1);
\filldraw[fill=black] (-1.6,-.3) circle (.1);
\filldraw[fill=white] (-2.1,-.3) circle (.1);
\end{tikzpicture}
\qq \qq
\begin{tikzpicture}[baseline=-0.35em,line width=0.7pt,scale=75/100]
\draw[thick] (-2.1,-.3) -- (-.6,-.3) -- (0,0) -- (-.4,.3) -- (-.9,.3);
\draw[thick] (0,0) -- (.5,0);
\filldraw[fill=black] (-.9,.3) circle (.1);
\filldraw[fill=black] (.5,0) circle (.1);
\filldraw[fill=black] (-.4,.3) circle (.1);
\filldraw[fill=black] (0,0) circle (.1);
\filldraw[fill=black] (-.6,-.3) circle (.1);
\filldraw[fill=black] (-1.1,-.3) circle (.1);
\filldraw[fill=black] (-1.6,-.3) circle (.1);
\filldraw[fill=white] (-2.1,-.3) circle (.1) node[left]{\scriptsize $i$};
\end{tikzpicture}
\\[2mm]
\begin{tikzpicture}[baseline=-0.35em,line width=0.7pt,scale=75/100]
\draw[thick] (.5,0) -- (1,0);
\draw[double,->] (1,0) -- (1.45,0);
\draw[thick] (1.5,0) -- (2,0);
\filldraw[fill=white] (.5,0) circle (.1);
\filldraw[fill=black] (1,0) circle (.1);
\filldraw[fill=black] (1.5,0) circle (.1);
\filldraw[fill=white] (2,0) circle (.1) node[right]{\scriptsize $i$};
\end{tikzpicture}
\qq \qq
\begin{tikzpicture}[baseline=-0.35em,line width=0.7pt,scale=75/100]
\draw[thick] (.5,0) -- (1,0);
\draw[double,->] (1,0) -- (1.45,0);
\draw[thick] (1.5,0) -- (2,0);
\filldraw[fill=white] (.5,0) circle (.1) node[left]{\scriptsize $i$};
\filldraw[fill=black] (1,0) circle (.1);
\filldraw[fill=black] (1.5,0) circle (.1);
\filldraw[fill=black] (2,0) circle (.1);
\end{tikzpicture}
\qq \qq \qq \qq
\begin{tikzpicture}[baseline=-0.35em,line width=0.7pt,scale=75/100]
\draw[double,->] (.9,0) -- (.95,0);
\draw[line width=.5pt] (.5,0) -- (1,0);
\draw[line width=.5pt] (.5,0.05) -- (1,0.05);
\draw[line width=.5pt] (.5,-0.05) -- (1,-0.05);
\filldraw[fill=white] (.5,0) circle (.1) node[left]{\scriptsize $i$};
\filldraw[fill=black] (1,0) circle (.1);
\end{tikzpicture} 
\qq \qq
\begin{tikzpicture}[baseline=-0.35em,line width=0.7pt,scale=75/100]
\draw[double,->] (.9,0) -- (.95,0);
\draw[line width=.5pt] (.5,0) -- (1,0);
\draw[line width=.5pt] (.5,0.05) -- (1,0.05);
\draw[line width=.5pt] (.5,-0.05) -- (1,-0.05);
\filldraw[fill=black] (.5,0) circle (.1);
\filldraw[fill=white] (1,0) circle (.1) node[right]{\scriptsize $i$};
\end{tikzpicture} 
\end{gather*}
\vspace{-1.5em}

\end{table}

Consider the real vector space $V = \R \Phi$.
For fixed $\theta \in \Autinv(\Phi)$ we have the decomposition $V = V^\theta \oplus V^{-\theta}$.
Denote by $\overline{\:\: \vphantom{\al}}: V \to V$ the corresponding projection onto $V^{-\theta}$.
Now consider the \emph{restricted Weyl group} and the set of \emph{restricted roots}
\eq{
\wb{W} = \{ w|_{V^{-\theta}} \,\, | \,\, w \in W, \, w(V^{-\theta}) \subseteq V^{-\theta} \}, \qq 
\wb{\Phi} = \{ \wb{\al} \,\, | \,\, \al \in \Phi \} \backslash \{0\}.
}
If $\theta = \theta(X,\tau)$ with $(X,\tau) \in \CD(A)$ then $W_X$ is a normal subgroup of $W^\theta = \{ w \in W \, | \, \theta w = w \theta \}$.
By \cite[Prop.~3.1]{He84} we have $\wb{W} \cong W^\theta / W_X$.
For $i \in I^*$ denote $X[i] = X \cup \{ i, \tau(i)\}$ and let $\wb{s}_i \in \GL(V^{-\theta})$ be the element that sends $\wb{\al}_i$ to $-\wb{\al}_i$ and fixes all $\beta \in V^{-\theta}$ with $\beta(h_i) = 0$.

\enlargethispage{1.2em} % this is to ensure the theorem stays on the same page

\begin{thrm}[\cite{He84} and \cite{Lu76}] \label{thrm:Heck}
Let $(X,\tau) \in \CD(A)$.
The following are equivalent:
\begin{enumerate}
\item We have $(X,\tau) \in \GSat(A)$.
\item For all $i \in I^*$, $\wb{s}_i \in \wb{W}$.
\item For all $i \in I^*$, $\wt s_i := w_X w_{X[i]}$ lies in $W^\theta$ and satisfies $\wt s_i|_{V^{-\theta}} = \wb{s}_i$.
\item For all $i \in I^*$, {$\tau_{0,X[i]}$} preserves $X$.
\item The restricted Weyl group $\wb{W}$ is the Weyl group of $\wb{\Phi}$.
\item The set $\{ \wt s_i \, | \, i \in I^* \}$ is a Coxeter system for the group it generates.
\end{enumerate}
\end{thrm}

\begin{proof}
The equivalence of the statements (ii), (iii), (iv) and (v) is shown in \cite[Lemma 3.2, Thm.~3.3, Thm.~4.4]{He84}.
The implication (iv) $\implies$ (vi) is shown in \cite[5.9 (i)]{Lu76} (also see \cite[25.1]{Lu02}).
Its converse follows by noting that if (iv) fails then for some $i \in I^*$, $w_{X[i]}$ and $w_{X}$ do not commute so that $\wt s_i^2 \ne \id_V$.
Finally, to show (i) $\Leftrightarrow$ (iv), note that by factorizability of $\tau_{0,X[i]}$ over connected components of $X[i]$, without loss of generality we may restrict to the case where $X[i]$ is connected and equals $I$.
Since there is nothing to prove if $\tau_{0,X[i]} = \id$, it remains to check the cases where $X[i]$ is of type ${\rm A}_n$ with $n>1$, ${\rm D}_n$ with $n>4$ odd or ${\rm E}_6$.
Classifying all $(X,\tau) \in \CD(A)$ such that $I = X[i]$ for some $i \in I^*$, $I$ is connected and $\tau_{0,X[i]} \ne \id$ we obtain the following diagrams in the top row:
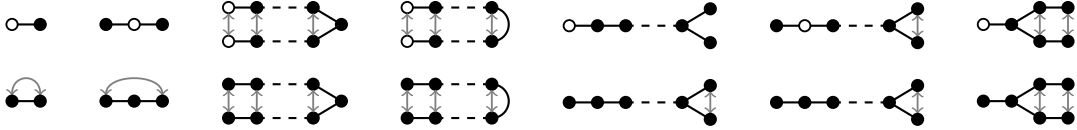
\[
\begin{array}{c>{\qu}c>{\qu}c>{\qu}c>{\qu}c>{\qu}c>{\qu}c}
\begin{tikzpicture}[baseline=-.3em,line width=0.7pt,scale=75/100]
\draw[thick] (0,0) -- (.5,0);
\filldraw[fill=white] (0,0) circle (.1);
\filldraw[fill=black] (.5,0) circle (.1);
\end{tikzpicture}
&
\begin{tikzpicture}[baseline=-.3em,line width=0.7pt,scale=75/100]
\draw[thick] (0,0) -- (1,0);
\filldraw[fill=black] (0,0) circle (.1);
\filldraw[fill=white] (.5,0) circle (.1);
\filldraw[fill=black] (1,0) circle (.1);
\end{tikzpicture}
&
\begin{tikzpicture}[baseline=-.3em,line width=0.7pt,scale=75/100]
\draw[thick] (0,.3) -- (.5,.3);
\draw[thick,dashed] (.5,.3) -- (1.5,.3);
\draw[thick] (1.5,.3) -- (2,0) -- (1.5,-.3);
\draw[thick] (0,-.3) -- (.5,-.3);
\draw[thick,dashed] (.5,-.3) -- (1.5,-.3);
\filldraw[fill=white] (0,.3) circle (.1);
\filldraw[fill=white] (0,-.3) circle (.1);
\draw[<->,gray] (0,.2) -- (0,-.2);
\filldraw[fill=black] (.5,.3) circle (.1);
\filldraw[fill=black] (.5,-.3) circle (.1);
\draw[<->,gray] (.5,.2) -- (.5,-.2);
\filldraw[fill=black] (1.5,.3) circle (.1);
\filldraw[fill=black] (1.5,-.3) circle (.1);
\draw[<->,gray] (1.5,.2) -- (1.5,-.2);
\filldraw[fill=black] (2,0) circle (.1);
\end{tikzpicture} 
&
\begin{tikzpicture}[baseline=-.3em,line width=0.7pt,scale=75/100]
\draw[thick] (0,.3) -- (.5,.3);
\draw[thick,dashed] (.5,.3) -- (1.5,.3);
\draw[thick,domain=270:450] plot({1.5+.3*cos(\x)},{.3*sin(\x)});
\draw[thick] (0,-.3) -- (.5,-.3);
\draw[thick,dashed] (.5,-.3) -- (1.5,-.3);
\filldraw[fill=white] (0,.3) circle (.1);
\filldraw[fill=white] (0,-.3) circle (.1);
\draw[<->,gray] (0,.2) -- (0,-.2);
\filldraw[fill=black] (.5,.3) circle (.1);
\filldraw[fill=black] (.5,-.3) circle (.1);
\draw[<->,gray] (.5,.2) -- (.5,-.2);
\filldraw[fill=black] (1.5,.3) circle (.1);
\filldraw[fill=black] (1.5,-.3) circle (.1);
\draw[<->,gray] (1.5,.2) -- (1.5,-.2);
\end{tikzpicture}
&
\begin{tikzpicture}[baseline=-0.25em,line width=0.7pt,scale=75/100]
\draw[thick] (2.5,0) -- (3.5,0);
\draw[thick,dashed] (3.5,0) -- (4.5,0);
\draw[thick] (5,.3) -- (4.5,0) -- (5,-.3);
\filldraw[fill=white] (2.5,0) circle (.1);
\filldraw[fill=black] (3,0) circle (.1);
\filldraw[fill=black] (3.5,0) circle (.1);
\filldraw[fill=black] (4.5,0) circle (.1);
\filldraw[fill=black] (5,.3) circle (.1);
\filldraw[fill=black] (5,-.3) circle (.1);
\end{tikzpicture}
&
\begin{tikzpicture}[baseline=-0.25em,line width=0.7pt,scale=75/100]
\draw[thick] (2.5,0) -- (3.5,0);
\draw[thick,dashed] (3.5,0) -- (4.5,0);
\draw[thick] (5,.3) -- (4.5,0) -- (5,-.3);
\draw[<->,gray] (5,.2) -- (5,-.2);
\filldraw[fill=black] (2.5,0) circle (.1);
\filldraw[fill=white] (3,0) circle (.1);
\filldraw[fill=black] (3.5,0) circle (.1);
\filldraw[fill=black] (4.5,0) circle (.1);
\filldraw[fill=black] (5,.3) circle (.1);
\filldraw[fill=black] (5,-.3) circle (.1);
\end{tikzpicture}
&
\begin{tikzpicture}[baseline=-.3em,line width=0.7pt,scale=75/100]
\draw[thick] (1,.3) -- (.5,.3) -- (0,0) -- (.5,-.3) -- (1,-.3);
\draw[thick] (0,0) -- (-.5,0);
\draw[<->,gray] (1,.2) -- (1,-.2);
\draw[<->,gray] (.5,.2) -- (.5,-.2);
\filldraw[fill=black] (1,.3) circle (.1);
\filldraw[fill=black] (1,-.3) circle (.1);
\filldraw[fill=black] (.5,.3) circle (.1);
\filldraw[fill=black] (.5,-.3) circle (.1);
\filldraw[fill=black] (0,0) circle (.1);
\filldraw[fill=white] (-.5,0) circle (.1);
\end{tikzpicture}
\\[4mm]
\begin{tikzpicture}[baseline=-.3em,line width=0.7pt,scale=75/100]
\draw[thick] (0,0) -- (.5,0);
\draw[<->,gray] (0,.1) .. controls (0,.5) and (.5,.5) .. (.5,.1);
\filldraw[fill=black] (0,0) circle (.1);
\filldraw[fill=black] (.5,0) circle (.1);
\end{tikzpicture}
&
\begin{tikzpicture}[baseline=-.3em,line width=0.7pt,scale=75/100]
\draw[thick] (0,0) -- (1,0);
\draw[<->,gray] (0,.1) .. controls (0,.5) and (1,.5) .. (1,.1);
\filldraw[fill=black] (0,0) circle (.1);
\filldraw[fill=black] (.5,0) circle (.1);
\filldraw[fill=black] (1,0) circle (.1);
\end{tikzpicture}
&
\begin{tikzpicture}[baseline=-.3em,line width=0.7pt,scale=75/100]
\draw[thick] (0,.3) -- (.5,.3);
\draw[thick,dashed] (.5,.3) -- (1.5,.3);
\draw[thick] (1.5,.3) -- (2,0) -- (1.5,-.3);
\draw[thick] (0,-.3) -- (.5,-.3);
\draw[thick,dashed] (.5,-.3) -- (1.5,-.3);
\filldraw[fill=black] (0,.3) circle (.1);
\filldraw[fill=black] (0,-.3) circle (.1);
\draw[<->,gray] (0,.2) -- (0,-.2);
\filldraw[fill=black] (.5,.3) circle (.1);
\filldraw[fill=black] (.5,-.3) circle (.1);
\draw[<->,gray] (.5,.2) -- (.5,-.2);
\filldraw[fill=black] (1.5,.3) circle (.1);
\filldraw[fill=black] (1.5,-.3) circle (.1);
\draw[<->,gray] (1.5,.2) -- (1.5,-.2);
\filldraw[fill=black] (2,0) circle (.1);
\end{tikzpicture} 
&
\begin{tikzpicture}[baseline=-.3em,line width=0.7pt,scale=75/100]
\draw[thick] (0,.3) -- (.5,.3);
\draw[thick,dashed] (.5,.3) -- (1.5,.3);
\draw[thick,domain=270:450] plot({1.5+.3*cos(\x)},{.3*sin(\x)});
\draw[thick] (0,-.3) -- (.5,-.3);
\draw[thick,dashed] (.5,-.3) -- (1.5,-.3);
\filldraw[fill=black] (0,.3) circle (.1);
\filldraw[fill=black] (0,-.3) circle (.1);
\draw[<->,gray] (0,.2) -- (0,-.2);
\filldraw[fill=black] (.5,.3) circle (.1);
\filldraw[fill=black] (.5,-.3) circle (.1);
\draw[<->,gray] (.5,.2) -- (.5,-.2);
\filldraw[fill=black] (1.5,.3) circle (.1);
\filldraw[fill=black] (1.5,-.3) circle (.1);
\draw[<->,gray] (1.5,.2) -- (1.5,-.2);
\end{tikzpicture}
&
\begin{tikzpicture}[baseline=-0.25em,line width=0.7pt,scale=75/100]
\draw[thick] (2.5,0) -- (3.5,0);
\draw[thick,dashed] (3.5,0) -- (4.5,0);
\draw[thick] (5,.3) -- (4.5,0) -- (5,-.3);
\draw[<->,gray] (5,.2) -- (5,-.2);
\filldraw[fill=black] (2.5,0) circle (.1);
\filldraw[fill=black] (3,0) circle (.1);
\filldraw[fill=black] (3.5,0) circle (.1);
\filldraw[fill=black] (4.5,0) circle (.1);
\filldraw[fill=black] (5,.3) circle (.1);
\filldraw[fill=black] (5,-.3) circle (.1);
\end{tikzpicture}
&
\begin{tikzpicture}[baseline=-0.25em,line width=0.7pt,scale=75/100]
\draw[thick] (2.5,0) -- (3.5,0);
\draw[thick,dashed] (3.5,0) -- (4.5,0);
\draw[thick] (5,.3) -- (4.5,0) -- (5,-.3);
\draw[<->,gray] (5,.2) -- (5,-.2);
\filldraw[fill=black] (2.5,0) circle (.1);
\filldraw[fill=black] (3,0) circle (.1);
\filldraw[fill=black] (3.5,0) circle (.1);
\filldraw[fill=black] (4.5,0) circle (.1);
\filldraw[fill=black] (5,.3) circle (.1);
\filldraw[fill=black] (5,-.3) circle (.1);
\end{tikzpicture}
&
\begin{tikzpicture}[baseline=-.3em,line width=0.7pt,scale=75/100]
\draw[thick] (1,.3) -- (.5,.3) -- (0,0) -- (.5,-.3) -- (1,-.3);
\draw[thick] (0,0) -- (-.5,0);
\draw[<->,gray] (1,.2) -- (1,-.2);
\draw[<->,gray] (.5,.2) -- (.5,-.2);
\filldraw[fill=black] (1,.3) circle (.1);
\filldraw[fill=black] (1,-.3) circle (.1);
\filldraw[fill=black] (.5,.3) circle (.1);
\filldraw[fill=black] (.5,-.3) circle (.1);
\filldraw[fill=black] (0,0) circle (.1);
\filldraw[fill=black] (-.5,0) circle (.1);
\end{tikzpicture}
\end{array}
\]
The bottom row shows the corresponding compatible decoration $(X[i],\tau_{0,X[i]})$.
The first case is not in $\GSat(A)$ and for the remaining six cases the subset $X$ is preserved by $\tau_{0,X[i]}$.
\end{proof}

\begin{rmk}
Note that $\wb{\Phi}$ is not always a root system.
By \cite[Thm.~6.1]{He84}, $\wb{\Phi}$ is a (possibly non-reduced or empty) root system precisely if $\tau_{0,X[i]}$ preserves $X$ for all $i \in I^*$ or $(X,\tau) = \begin{tikzpicture}[baseline=-0.3em,line width=0.7pt,scale=75/100]
\draw[thick] (.5,0) -- (1,0);
\filldraw[fill=white] (.5,0) circle (.1);
\filldraw[fill=black] (1,0) circle (.1);
\end{tikzpicture}$. 
\hfill \rmkend
\end{rmk}

%%%%%%%%%%%%%%%%%%%%%%%%%%%%%%%%%%%%%%%%%%%%%%%%%%%%%%%%%%%%%%%%%%%%%%%%%%%%%%
% Section 3
%%%%%%%%%%%%%%%%%%%%%%%%%%%%%%%%%%%%%%%%%%%%%%%%%%%%%%%%%%%%%%%%%%%%%%%%%%%%%%

\section{The subalgebra $\mfk$} \label{sec:k}

For $(X,\tau) \in \Sat(A)$ the subalgebra $\mfg^\theta$ can be described in terms of generators; see e.g.~\cite[Lemma 2.8]{Ko14}.
This motivates the following more general definition.

\begin{defn} \label{defn:k}
Let $(X,\tau) \in \CD(A)$.
For $\bm \ga \in \C^{I \backslash X}$ let $\mfk_{\bm \ga} = \mfk_{\bm \ga}(X,\tau)$ be the Lie subalgebra of $\mfg$ generated by $\mfg_X$, $\mfh^\theta$ and, for all $i \in I \backslash X$,
\begin{flalign}
\label{bi:def} 
&& b_{i;\ga_i} := f_i + \ga_i \,\theta(f_i).
&& \defnend
\end{flalign}
\end{defn}

It is convenient to suppress the dependence on $\bm \ga$ and simply write $b_i$ and $\mfk$ if there is no cause for confusion.
We denote $b_i = f_i$ if $i \in X$.
Note that $\mfh_X \subseteq \mfh^\theta$.
It follows that $\mfk$ is generated by $\mfn^+_X := \{ e_i \, | \, i \in X\}$, $\mfh^\theta$ and $b_i$ for $i \in I$.
Owing to \eqrefs{g:rel1}{g:Serre} and \eqref{theta:basic}, these satisfy
\ali{
\label{k:rels1} [e_i,b_j] &= \del_{ij} h_i \in \mfh^\theta &\qq & \text{for all } i \in X, \, j \in I, \\
\label{k:rels2} [h,b_j] &= - \al_j(h)\, b_j && \text{for all } h \in \mfh^\theta, \, j \in I, \\
\label{k:rels3} [h,e_j] &= \al_j(h)\, e_j && \text{for all } h \in \mfh^\theta, \, j \in X, \\
\label{k:rels4} [h,h'] &= 0 && \text{for all }h,h' \in \mfh^\theta, \\
\label{k:rels5} \ad(e_i)^{1-a_{ij}}(e_j) &= 0 && \text{for all } i,j \in X, \, i \ne j.
}
In Appendix \ref{appendix} we study the repeated adjoint action of $b_i$ on $b_j$ for $i,j \in I$ such that $i \ne j$.
By setting $m=M_{ij}$ in Lemmas \eqrefs{lem:biij:a}{lem:biij:c} one obtains the following Serre-type relations for the generators $b_i$:
\eq{  \label{k:Serre}
\hspace{-3pt} \ad(b_i)^{M_{ij}}(b_j) = 
\begin{cases}
(1+\zeta(\al_i)) \, \ga_i\, [\theta(f_i),[f_i,f_j]] \in \mfn^+_X  \hspace{-5pt}
& \text{if } \theta(\al_i)+\al_i+\al_j \in \Phi^-, \, a_{ij}=-1, \hspace{-5pt} \\
-18 \ga_i^2\, e_j 
& \text{if } \theta(\al_i)+\al_i+\al_j=0, \, a_{ij}=-3,  \\
-\ga_i\, (2h_i+h_j)
& \text{if } \theta(\al_i)+\al_i+\al_j=0, \, a_{ij}=-1,  \\
\big( \ga_i +  \zeta(\al_i) \ga_j \big)\, [\theta(f_i),f_j] \in \mfn^+_X \hspace{-5pt}
& \text{if } \theta(\al_i)+\al_j \in \Phi^-, \,a_{ij}=0, \\
\ga_j h_i - \ga_i\, h_j
& \text{if } \theta(\al_i)+\al_j=0, \, a_{ij}=0, \\
2\,(\ga_i+\ga_j)\, b_i 
& \text{if } \theta(\al_i)+\al_j=0, \, a_{ij}=-1, \\
\displaystyle\sum_{r=1}^{\big\lfloor \tfrac{M_{ij}}{2} \big\rfloor} p_{ij}^{(r,M_{ij})} \ga_i^r \ad(b_i)^{M_{ij}-2r}(b_j) \hspace{-5pt}
& \text{if } \theta(\al_i)+\al_i=0, \, j \in I \backslash X, \\
0 & \text{otherwise}.
\end{cases} 
\hspace{-5pt} 
}
For $i,j \in I$ such that $i \ne j$ and $m,r \in \Z_{\ge 0}$ we have defined $p_{ij}^{(r,m)} \in \Z$ by setting $p_{ij}^{(r,m)} = 0$ if $r> \lfloor m/2 \rfloor$, $p_{ij}^{(0,m)} = -1$ and
\eq{ \label{pij:def} 
p_{ij}^{(r,m)} = p_{ij}^{(r,m-1)} +  (m-1) (M_{ij}+1-m)\, p_{ij}^{(r-1,m-2)} \qq \text{if } 0<r \le \lfloor m/2 \rfloor.
}
By induction with respect to $m$, it can be shown that these integers satisfy
\eq{ \label{pij:negative}
p_{ij}^{(r,m)} <0 \qq \text{if } 0 \le 2r \le m \le M_{ij}.
}
Indeed, \eqref{pij:negative} is true for $m\in \{0,1\}$.
Suppose $0 \le 2r \le m$ and $1< m \le M_{ij}$ and assume \eqref{pij:negative} holds with $m$ replaced by $m-1$ and by $m-2$.
If $p_{ij}^{(r,m-1)} = 0$ we must have $r=m/2$ so that $p_{ij}^{(r-1,m-2)} < 0$.
Hence at least one of $p_{ij}^{(r,m-1)}, p_{ij}^{(r-1,m-2)}$ is nonzero and the observation that $(m-1) (M_{ij}+1-m) > 0$ completes the induction step.

As $\mfg$ is of finite type, $M_{ij} \in \{ 1,2,3,4\}$.
Hence the penultimate case of \eqref{k:Serre} amounts to
\[
\ad(b_i)^{M_{ij}}(b_j) = \begin{cases}
0 & \text{if } a_{ij}=0 \\
-\ga_i b_j & \text{if } a_{ij} = -1 \\
- 4 \ga_i\, [b_i,b_j] & \text{if } a_{ij}=-2 \\
- 9 \ga_i^2 b_j - 10 \ga_i\, [b_i,[b_i,b_j]] & \text{if } a_{ij}=-3
\end{cases}
\qu \text{and } \theta(\al_i)+\al_i = 0, \, j \in I \backslash X.
\]

\begin{rmk} \mbox{}
\begin{enumerate}
\item 
Definition \ref{defn:k} can be used in the general Kac-Moody case, so that \eqrefs{k:rels1}{k:rels5} still hold.
Also the results of Appendix \ref{appendix} are valid in this general setting and hence so is \eqref{k:Serre}.
We will discuss the subalgebra $\mfk(X,\tau)$ in the Kac-Moody setting in future work.

\item
The relations \eqref{k:Serre} entail that $\ad(b_i)^{M_{ij}}(b_j) = 0$ if $i \in X$ or if $\tau(i) \notin \{ i,j \}$ as required by the specialization of \cite[Eq.~(5.20) and Thm.~7.3]{Ko14}.
\hfill \rmkend
\end{enumerate}
\end{rmk}

%%%%%%%%%%%%%%%%%%%%%%%%%%%%%%%%%%%%%%%%%%%%%%%%%%%%%%%%%%%%%%%%%%%%%%%%%%%%%%

\subsection{Basic structure of $\mfk$}

From now on we assume that the $\ga_i$ are nonzero.
In order to state the main result of this section, we need some notation.
For all $i,j \in I$ such that $i \ne j$ denote $\la_{ij} := M_{ij}\, \al_i + \al_j \in Q^+ \backslash \Phi$ and consider the sets
\begin{align}
\Idiff &= \{ i \in I^* \,\,|\,\,  i \ne \tau(i) \text{ and } (\theta(\al_i))(h_i) \ne 0 \} \\
\nonumber &=  \{ i \in I^* \,\,|\,\,  i \ne \tau(i) \text{ and } \exists j \in X[i] {\rm \,\;s.t.\;} a_{ij} < 0 \} \\
\Ga &= \Ga(X,\tau) =  \{ \bm \ga \in (\C^\times)^{I \backslash X} \,\, | \,\, \ga_i = \ga_{\tau(i)}  \text{ if } i \in I^* \backslash \Idiff \}.
\end{align}
For $\bm i \in I^\ell$ with $\ell \in \Z_{>0}$ we write $\al_{\bm i} = \sum_{r=1}^\ell \al_{i_r}$ and
\[ 
b_{\bm i} = \ad(b_{i_1}) \cdots \ad(b_{i_{\ell-1}}) (b_{i_\ell}), \qu
e_{\bm i} = \ad(e_{i_1}) \cdots \ad(e_{i_{\ell-1}}) (e_{i_\ell}), \qu
f_{\bm i} = \ad(f_{i_1}) \cdots \ad(f_{i_{\ell-1}}) (f_{i_\ell}).
\]
Observe that 
\[
\mfn^- = \Sp \{ f_{\bm i} \, | \, \bm i \in I^\ell, \, \ell>0 \}, \qq \qq \mfn^+_X = \Sp \{ e_{\bm i} \, | \, \bm i \in X^\ell, \, \ell>0 \}.
\]
Hence for all $\ell \in \Z_{>0}$ we can choose $\mc{J}_\ell \subseteq I^\ell$ such that $\{ f_{\bm i} \}_{\bm i \in \mc{J}_\ell}$ is a basis for $\Sp \{ f_{\bm i} \}_{\bm i \in I^\ell}$ and $\{ e_{\bm i} \}_{\bm i \in \mc{J}_{X,\ell}}$ is a basis for $\Sp \{ e_{\bm i} \}_{\bm i \in X^\ell}$ where $\mc{J}_{X,\ell} := \mc{J}_\ell \cap X^\ell$.
Let
\eq{ \label{J:def}
\mc{J} := \bigcup_{\ell \in \Z_{>0}} \mc{J}_\ell, \qq \qq \mc{J}_X := \bigcup_{\ell \in \Z_{>0}} \mc{J}_{X,\ell}.
}
Then $\{ f_{\bm i} \}_{\bm i \in \mc{J}}$ is a basis of $\mfn^-$ and $\{ e_{\bm i} \}_{\bm i \in \mc{J}_X}$ is a basis of $\mfn^+_X$.

\begin{thrm} \label{mainthrm}
Let $(X,\tau) \in \CD(A)$ and $\bm \ga \in (\C^\times)^{I \backslash X}$.
The following statements are equivalent:
\begin{enumerate}
\item We have $(X,\tau) \in \GSat(A)$ and $\bm \ga \in \Ga$.
\item For all $i,j \in I$ such that $i \ne j$ we have
\eq{ \label{k:goodSerre}
\ad(b_i)^{M_{ij}}(b_j) \in \mfn^+_X \oplus \mfh^\theta \oplus \bigoplus_{\ell \in \Z_{> 0}} \bigoplus_{\bm k \in I^\ell, \, \al_{\bm k} < \la_{ij}} \C b_{\bm k}.
}
\item We have the following identity for $\mfh^\theta$-modules:
\eq{ \label{k:tridecomp} 
\mfk  = \mfn^+_X \oplus \mfh^\theta \oplus \bigoplus_{\bm i \in \mc{J}} \C b_{\bm i}.
}
\item We have 
\eq{ \label{k:intersection}
\mfk \cap \mfh = \mfh^\theta.
}
\end{enumerate}
\end{thrm}

\begin{proof} \mbox{}
\begin{description}
\item[\rm (i) \;$\Longleftrightarrow$ (ii)]
This is a direct consequence of \eqref{k:Serre}.
\item[\rm (ii) $\implies$ (iii)]
Owing to \eqrefs{k:rels2}{k:rels4} it is sufficient to prove \eqref{k:tridecomp} as an identity for vector spaces.
First we prove that $\mfk =  \mfn^+_X + \mfh^\theta + \Sp\{ b_{\bm i} \, | \bm i \in \mc{J} \}$.
From \eqrefs{k:rels1}{k:rels2} it follows that
\eq{ \label{k:sum} 
\mfk = \mfn^+_X + \mfh^\theta + \langle b_j \rangle_{j \in I} = \mfn^+_X + \mfh^\theta + \sum_{\ell \in \Z_{>0}} \sum_{\bm i \in I^\ell} \C b_{\bm i}
}
as vector spaces.
Hence it suffices to prove that for all $\bm j \in \cup_\ell I^\ell$ we have
\eq{ 
\label{eqn:subspacesum} b_{\bm j} \in \mfn^+_X + \mfh^\theta + \Sp\{ b_{\bm i} \,\,|\,\, \bm i \in \mc{J} \}.
}
We will prove this by induction with respect to the height $\ell$.
Since for all $j \in I$ we have $\dim(\mfg_{-\al_j})=1$ and hence $(j) \in \mc{J}$, the case $\ell=1$ is trivial.
Now fix $\ell \in \Z_{>1}$ and assume that \eqref{eqn:subspacesum} holds true for all smaller positive integers.
Fix $\bm j \in I^\ell$ and repeatedly apply the Serre relations \eqref{g:Serre} to obtain that for all $\bm i \in \mc{J}_\ell$ there exist $a_{\bm i} \in \C$ such that $f_{\bm j} = \sum_{\bm i \in \mc{J}_\ell} a_{\bm i} f_{\bm i}$.
Hence, by virtue of (ii) and equations \eqrefs{k:rels1}{k:rels2} it follows that
\eq{
b_{\bm j} - \sum_{\bm i \in \mc{J}_\ell} a_{\bm i} b_{\bm i} \in \mfn^+_X + \mfh^\theta + \Sp\bigg\{ b_{\bm i} \,\bigg\vert\, \bm i \in \bigcup_{m=1}^{\ell-1} I^m  \bigg\}.
}
Using the induction hypothesis for the elements $b_{\bm i}$ in the last summation one obtains \eqref{eqn:subspacesum}.
It remains to show that the sum in \eqref{eqn:subspacesum} is direct.
Let $\bm j \in \mc{J}$.
Then $f_{\bm j}$ is nonzero.
Because of the explicit formula \eqref{bi:def} we have
\eq{ \label{bi:weight}
b_{\bm j} - f_{\bm j} \in  \mfn^+_X + \mfh^\theta + \C \theta(f_{\bm j}) + \Sp\{ b_{\bm i} \,\,|\, \bm i \in \mc{J},\, \al_{\bm i} < \al_{\bm j} \}.
}
Hence $f_{\bm j} = \pi_{-\al_{\bm j}}(b_{\bm j})$ for all $\bm j \in \mc{J}$, where $\pi_\al$ is the projection on $\mfg_\al$ for $\al \in \Phi$, see \eqref{g:rootdecomp}.
Thus the linear independence of $\{ f_{\bm j} \}_{\bm j \in \mc{J}}$ together with \eqref{g:rootdecomp} implies that the sum is direct.

\item[\rm (iii) $\implies$ (iv)]
By definition, $\mfh^\theta \subseteq \mfk \cap \mfh$ so it suffices to show that $\mfk \cap \mfh \subseteq \mfh^\theta$.
Suppose $h \in \mfk \cap \mfh^\theta$.
By $\pi_{-\al_{\bm j}}(b_{\bm j}) = f_{\bm j}$ and the triangular decomposition \eqref{g:rootdecomp}, part (iii) implies $h \in \mfn^+_X \oplus \mfh^\theta \subseteq \mfg^\theta$ so $h \in \mfh^\theta$.

\item[\rm (iv) \,$\implies$ (ii)]
We prove the contrapositive.
If \eqref{k:goodSerre} fails then \eqref{k:Serre} and \eqref{htheta:expl} imply
\eq{
\ga_j h_i - \ga_i h_j \in \mfk \cap (\mfh \backslash \mfh^\theta) \text{ with } \ga_i \ne \ga_j \qq \text{or} \qq 2h_i+h_j \in \mfk \cap (\mfh \backslash \mfh^\theta).
}
In either case \eqref{k:intersection} does not hold.
\qedhere
\end{description}
\end{proof}

Note that if $\mfk = \mfg^{\theta_{\bm \ga}}$ then \eqref{k:intersection} is trivially satisfied, since $\mfh^\theta = \mfh^{\theta_{\bm \ga}}$.
Given $(X,\tau) \in \GSat(A)$, $\bm \ga \in \Ga$ and $\mc{J}$ defined by \eqref{J:def}, from \eqref{htheta:expl} and \eqref{k:tridecomp} we obtain the \emph{standard basis} for $\mfk$:
\eq{
\label{k:basis}
\{ e_{\bm i} \, | \, \bm i \in \mc{J}_X \} \; \cup \; \{ h_i \, | \, i \in X \} \; \cup \; \{ h_i-h_{\tau(i)} \, | \, i \in I^*, \, i \ne \tau(i)\} \; \cup \; \{ b_{\bm i} \, | \, \bm i \in \mc{J} \}.
}
We denote $\Phi_X = \Phi \cap Q_X$ where $Q_X = \sum_{i \in X} \Z \al_i$.
Combining \eqref{k:basis} with $|\mc{J}| = | \Phi |/2$ and $\dim(\mfh^\theta) = |I|-|I^*|$, itself a consequence of \eqref{htheta:expl}, it follows that
\eq{ \label{k:dimension}
\dim(\mfk) = |\Phi_X|/2 + |I| - |I^*| + |\Phi|/2.
}

J.~Stokman showed in \cite{St19} that \emph{generalized Onsager algebra}, i.e.~the Lie algebra with generators $\wt b_i$ ($i \in I$) and relations
\eq{
\ad(\wt b_i)^{M_{ij}}(\wt b_j) = \sum_{r=1}^{\lfloor M_{ij}/2 \rfloor} p_{ij}^{(r,M_{ij})} \ad(\wt b_i)^{M_{ij}-2r}(\wt b_j)
}
is isomorphic to $\mfk_{(1,1\ldots,1)}(\emptyset,\id) = \mfg^{\om}$ via $\wt b_i \mapsto b_i = f_i + \om(f_i) = f_i - e_i$.
Without loss of generality we can set $\bm \ga = (1,1,\ldots,1)$ since $\mfk_{\bm \ga} = \Ad(\chi{(\ga_1^{1/2},\ga_2^{1/2},\ldots,\ga_{|I|}^{1/2})})(\mfk_{(1,1,\ldots,1)})$ for all $\bm \ga \in (\C^\times)^{I \backslash X}$.
We now discuss a generalization of this to arbitrary $(X,\tau) \in \GSat(A)$.

\begin{conj} \label{conj2}
Let $(X,\tau) \in \GSat(A)$ and $\bm \ga \in \Ga$.
The Lie algebra $\wt \mfk$ generated by symbols $\wt h_i, \wt e_i$ ($i \in X$), $\wt{h_i-h_{\tau(i)}}$ ($i \in I^*, i \ne \tau(i)$), $\wt b_i$ ($i \in I$) and the relations obtained from \eqrefs{k:rels1}{k:Serre} by adding tildes appropriately is isomorphic to $\mfk$.
\end{conj}

The only obstacle to promoting this to a theorem, and thereby settling the question posed in \cite[Rmk.~2.10]{Ko14}, is the lack of a general explicit formula for the right-hand sides in \eqref{k:Serre} in terms of the $e_k$ with $k \in X$, instead of $\theta(f_i)$;
for individual choices of $(X,\tau) \in \GSat(A)$ these explicit expressions can be found, or such relations do not occur at all, and in those cases one could prove the statement in the conjecture as follows.
Because the generators of the Lie subalgebra $\mfk$ satisfy \eqrefs{k:rels1}{k:Serre}, one has a surjective Lie algebra homomorphism $\phi: \wt \mfk \to \mfk$ defined on generators by removing the tilde.
On the other hand, in $\mfk$ there are no relations involving the $b_i$ other than \eqrefs{k:rels1}{k:Serre}, as otherwise applying the appropriate projection $\pi_{-\al}$ onto $\mfg_{-\al}$ with $\al \in \Phi^+$ maximal would yield a relation involving the $f_i$ other than those given in \eqrefs{g:rel1}{g:Serre}.
From this one can deduce that $\phi$ is injective and obtain the statement in Conjecture \ref{conj2}.

%%%%%%%%%%%%%%%%%%%%%%%%%%%%%%%%%%%%%%%%%%%%%%%%%%%%%%%%%%%%%%%%%%%%%%%%%%%%%%

\subsection{Semidirect product decompositions of $\mfk$}

In this section we assume that $A$ is indecomposable, so that $\mfg$ is simple.
In order to describe the derived subalgebra of $\mfk$ recall the set $\Idiff \in I^*$ and define
\eqa{ \label{specialorbits}
I_{\rm ns} &= \{ i \in I \, | \, (\theta(\al_i))(h_i) = -2 \} = \{ i \in I \, | \, i = \tau(i), \, \check X(i) = \emptyset \} \subseteq I^* , \\
\Insf &= \{ j \in I_{\rm ns} \, | \, a_{ij} \in 2\Z \text{ for all } i \in I_{\rm ns} \}. 
}

\begin{prop} \label{prop:kprime}
Let $(X,\tau) \in \GSat(A)$ and $\bm \ga \in \Ga$.
The set
\[
\{ e_{\bm i} \, | \, \bm i \in \mc{J}_X \} \; \cup \; 
\{ h_i \, | \, i \in X \} \; \cup \; 
\{ h_i-h_{\tau(i)} \, | \, i \in I^* \backslash \Idiff, \, i \ne \tau(i)\} \; \cup \; 
\{ b_{\bm i} \, | \, \bm i \in \mc{J}, \, \bm i \ne (j) \text{ with } j \in \Insf \}.
\]
is a basis for the derived subalgebra $\mfk'$ and we have
\eq{ \label{kprime}
\mfk = \mfk' \rtimes \Big( \bigoplus_{i \in \Idiff} \C (h_i - h_{\tau(i)}) \oplus \bigoplus_{j \in \Insf} \C b_j \Big).
}
\end{prop}

\begin{proof}
Fix $(X,\tau) \in \GSat(A)$.
Note that in \eqrefs{k:rels1}{k:Serre} neither $h_i-h_{\tau(i)}$ ($i \in \Idiff$) nor $b_j$ ($j \in \Insf$) appears in the right-hand side.
From the decomposition \eqref{k:tridecomp} it follows that these elements are not linear combinations of Lie brackets in $\mfk$.
It suffices to show that the remaining basis elements specified in \eqref{k:basis} are linear combinations of Lie brackets in $\mfk$.
\begin{itemize}
\item For $b_{\bm i}$ with $\bm i \in \mc{J}_\ell$ and $e_{\bm i}$ with $\bm i \in \mc{J}_{X,\ell}$ with $\ell>1$, this holds by definition.
\item For $e_i$, $f_i$, $h_i$ with $i \in X$, this follows from \eqrefs{k:rels1}{k:rels3}.
\item For $h_i-h_{\tau(i)}$ with $i \in I^* \backslash \Idiff$ and $i \ne \tau(i)$, the constraint on $i$ is equivalent to $w_X(\al_i)=\al_i$ and $a_{i \, \tau(i)}=0$.
Hence \eqref{k:Serre} implies that $h_i-h_{\tau(i)} = \ga_i^{-1} [b_i,b_{\tau(i)}]$.
\item For $b_j$ with $X(j) \ne \emptyset$ there exists $i \in X$ such that $a_{ij} \ne 0$.
By \eqref{k:rels2} we have $b_j = -a_{ij}^{-1}[h_i,b_j]$.
\item For $b_j$ with $j \ne \tau(j)$, by \eqref{k:rels2} we have $b_j =  (a_{\tau(j) \, j}-2)^{-1} [h_j-h_{\tau(j)},b_j]$.
\item For $b_j$ with $j \in I_{\rm ns} \backslash \Insf$ there exists $i \in I_{\rm ns}$ such that $a_{ij}$ is odd.
From \eqref{k:Serre} we deduce
\eq{
p_{ij}^{(M_{ij},2M_{ij})} b_j = \ga_i^{-M_{ij}} \ad(b_i)^{2M_{ij}}(b_j) - \sum_{r=1}^{M_{ij}-1} p_{ij}^{(r,2M_{ij})} \ga_i^{r-M_{ij}} \ad(b_i)^{2(M_{ij}-r)}(b_j).
}
By \eqref{pij:negative} $p_{ij}^{(M_{ij},2M_{ij})}$ is nonzero.
\hfill \qedhere
\end{itemize}
\end{proof}

From Proposition \ref{prop:kprime} it follows that the codimension of $\mfk'$ in $\mfk$ equals $|\Idiff| + |\Insf|$.
For $(X,\tau) \in \Sat(A)$, in \cite[Sec.~7, Variation 1]{Le02} it was noted that $|\Idiff| \le 1$ if $A$ is of finite type.
In light of the above it is natural to generalize this by involving the set $\Insf$ and allowing $(X,\tau) \in \GSat(A)$.
Namely we have $|\Idiff| + |\Insf| \le 1$ for all $(X,\tau) \in \GSat(A)$ and the upper bound is sharp unless $A$ is of type ${\rm E}_8$, ${\rm F}_4$ or ${\rm G}_2$.
This extends the known result for involutive $\theta$ that $\mfg^\theta$ is reductive with abelian summand at most 1-dimensional, see \cite{Ar62}.

\begin{defn} 
The set of \emph{weak Satake diagrams} is
\eq{
\WSat(A) = \GSat(A) \backslash (\Sat(A) \cup \{ \nonweak \}).
}
\end{defn}

For $(X,\tau) \in \WSat(A)$ we will obtain a semidirect product decomposition in terms of a reductive Lie subalgebra and a nilpotent ideal.
For any $r \in \Z_{\ge 0}$ and any $i \in I$ denote by $\mfk(i)_r$ the span of all $b_{\bm j}$ with $\bm j \in \mc J$ such that the coefficient of $\al_i$ in $\al_{\bm j}$ is at least $r$.
Set $\mfk(i) := \mfk(i)_1 \subseteq \mfk$ and recall the $\bm \ga$-modified automorphism $\theta_{\bm \ga}$ defined in \eqref{theta_gamma}.

\begin{prop} \label{prop:k:weak:ideal}
Let $(X,\tau) \in \WSat(A)$, $\bm \ga \in \Ga$ and $i$ the unique element of $I \backslash X$ such that $i = \tau(i)$ and $\al_i(\rho^\vee_X) \notin \Z$.
Then the following statements hold.
\begin{enumerate}
\item The subalgebra $\mfg_{I \backslash \{ i \}}$ is $\theta_{\bm \ga}$-stable, $\theta_{\bm \ga}|_{\mfg_{I \backslash \{ i \}}}$ is an involution and $\mfk_{\hat \imath} := \mfk \cap \mfg_{I \backslash \{ i \}}$ is its fixed-point subalgebra in $\mfg_{I \backslash \{ i \}}$;
\item We have $\ad(b_i)(\mfk(i)_r) \subseteq \mfk(i)_{r+1}$ for all $r \in \Z_{\ge 0}$ and the subspaces $\mfk(i)_r$ are $\ad(\mfk_{\hat \imath})$-modules;
\item The subspace $\mfk(i)$ is an ideal of $\mfk$, $\mfk = \mfk(i) \rtimes \mfk_{\hat \imath}$ and we have the lower central series
\[ 
\mfk(i) = \mfk(i)_1 \supset \mfk(i)_2 \supset \mfk(i)_3 = \mfk(i)_4 = \ldots = \{ 0 \};
\]
\item The subalgebra $\mfk_{\bm \ga}$ is isomorphic to the subalgebra of $\mfg$ generated by $\mfg_X$, $\mfh^\theta$, $b_{j;\gamma_j}$ for $j \in I \backslash (X \cup \{ i \})$ and $b_{i;0} = f_i$.
\end{enumerate}
\end{prop}

\begin{proof}
From the decomposition \eqref{k:tridecomp} it follows that $\mfk_{\hat \imath} = \langle \mfn_X^+, \mfh^\theta, \mfk(i)_0 \rangle$ and, as vector spaces, $\mfk = \mfk(i) \oplus \mfk_{\hat \imath}$.
Crucially, by \eqref{k:Serre} we have $\ad(b_i)^{M_{ij}}(b_j) = 0$ for all $j \in I \backslash \{i \}$, since $\al_i(\rho^\vee_X) \notin \Z$.
Now we prove the four statements consecutively.
\begin{enumerate}
\item Since $\theta_{\bm \ga}(\mfg_\al) = \mfg_{\theta(\al)}$ for all $\al \in \Phi$ and applying $\theta = -w_X \tau$ to any root of $\mfg_{I \backslash \{i\}}$ does not modify the coefficient of $\al_i$, it follows that $\mfg_{I \backslash \{i\}}$ is $\theta_{\bm \ga}$-stable.
Note that $\mfk_{\hat \imath}$ is fixed pointwise by $\theta_{\bm \ga}$.
Furthermore a dimension count in each simple summand of $\mfg_{I \backslash \{i\}}$ combined with \eqref{k:dimension} implies that $\mfk_{\hat \imath}$ is the fixed-point subalgebra of $\theta_{\bm \ga}$.

\item The first statement is immediate.
From the adjoint action of $e_j$ ($j \in X$), $h \in \mfh^\theta$ and $b_j$ ($j \in I \backslash \{i \}$) on elements of $\mfk(i)_r$, subject to \eqrefs{k:rels1}{k:Serre}, we obtain that $\ad(\mfk_{\hat \imath})(\mfk(i)_r) \subseteq \mfk(i)_r$.

\item From part (ii) it follows that $[\mfk(i),\mfk] \subset \mfk(i)$ and combining this with $\mfk = \mfk(i) \oplus \mfk_{\hat \imath}$ we obtain the semidirect product decomposition.
A case-by-case analysis using Table \ref{tab:nonSatake} yields that the coefficient in front of $\al_i$ in the highest root of $\Phi$ is always 2.
This implies that the lower central series becomes trivial after 2 steps.

\item This follows from the facts that the relations involving $b_{i;\ga_i}$ in \eqrefs{k:rels1}{k:Serre} do not depend on $\ga_i$ and that the derivation of \eqref{k:Serre} did not require $\ga_i \ne 0$.
\hfill \qedhere
\end{enumerate}
\end{proof}

\begin{exam} \label{ex:sp4g2} 
We discuss two examples of $\mfk(X,\tau)$ with $(X,\tau) \in \GSat(A) \backslash \Sat(A)$.
\begin{enumerate}
\item
The smallest such $\mfk$ occurs when $(X,\tau)=$\begin{tikzpicture}[baseline=-0.35em,line width=0.7pt,scale=75/100]
\draw[double,<-] (4.05,0) --  (4.4,0);
\filldraw[fill=white] (4,0) circle (.1) node[left=.25pt]{\scriptsize $1$};
\filldraw[fill=black] (4.5,0) circle (.1) node[right=.25pt]{\scriptsize $2$};
\end{tikzpicture}.
By definition, $\mfk$ is the subalgebra of $\mfsp_4$ generated by $b_1 = f_1 + \ga_1 \theta(f_1)$ for some $\ga_1\in\C^\times$ and $b_2=f_2$, $e_2$, $h_2$.
The relations \eqrefs{k:rels1}{k:Serre} specialize to
\eq{ \label{ex:sp4:relations}
\begin{gathered}
[e_2,b_1]=0, \qq [e_2,b_2]=h_2, \qq [h_2,b_1]=b_1, \qq [h_2,b_2]=-2b_2, \\
[h_2,e_2]=2e_2, \qq [b_1,[b_1,[b_1,b_2]]]=0, \qq [b_2,[b_2,b_1]]=0.
\end{gathered}
}
According to \eqref{k:basis}, the standard basis of $\mfk$ is given by $\{ e_2, h_2, b_1, b_2, b_{(1,2)}, b_{(1,1,2)} \}$.
Proposition \ref{prop:kprime} implies $\mfk = \mfk'$ and Proposition \ref{prop:k:weak:ideal} yields the nontrivial Levi decomposition $\mfk = \Sp(b_1,b_{(1,2)},b_{(1,1,2)}) \rtimes \Sp(e_2,h_2,b_2)$ with the radical isomorphic to the 3-dimensional Heisenberg Lie algebra and the Levi subalgebra isomorphic to $\mfsl_2$.
In particular it follows from \eqref{ex:sp4:relations} that $b_{(1,1,2)}$ is central.
\item
Proposition \ref{prop:k:weak:ideal} excludes the case $(X,\tau) =$ \begin{tikzpicture}[baseline=-0.35em,line width=0.7pt,scale=75/100]
\draw[double,->] (.9,0) -- (.95,0);
\draw[line width=.5pt] (.5,0) -- (1,0);
\draw[line width=.5pt] (.5,0.05) -- (1,0.05);
\draw[line width=.5pt] (.5,-0.05) -- (1,-0.05);
\filldraw[fill=black] (.5,0) circle (.1) node[above=.25pt]{\scriptsize $1$};
\filldraw[fill=white] (1,0) circle (.1) node[above=.25pt]{\scriptsize $2$};
\end{tikzpicture}
.
We will now see that is the only element of $\GSat(A)\backslash\Sat(A)$ such that $\mfk$ is a reductive Lie algebra.
By definition, $\mfk$ is the subalgebra of $\mfg={\rm Lie}(G_2)$ generated by $e_1,h_1,b_1=f_1$ and $b_2=f_2+\ga_2\,\theta(f_2)$ for some $\ga_2 \in \C^\times$.
The relations \eqrefs{k:rels1}{k:Serre} give
\eq{ \label{ex:g2:relations}
\begin{gathered}
[e_1,b_1]=h_1, \qq  [e_1,b_2]=0, \qq [h_1,b_1] = -2b_1, \qq [h_1,b_2] = b_2, \\ 
[h_1,e_1]=2e_1, \qq [b_1,[b_1,b_2]] =0 , \qq  [b_2,[b_2,[b_2,[b_2,b_1]]]] = -18 \ga_2^2 e_1.
\end{gathered}
}
The standard basis of $\mfk$ is given by $\{ e_1, h_1, b_1, b_2, b_{(2,1)}, b_{(2,2,1)}, b_{(2,2,2,1)}, b_{(1,2,2,2,1)} \}$.
Proposition \ref{prop:kprime} yields $\mfk = \mfk'$.
Moreover,
using \eqref{ex:g2:relations}, the adjoint action of $e_1$, $b_1$ and $b_2$ on $\mfk$ implies that any ideal of $\mfk$ equals $\mfk$ if it contains any of the above standard basis elements.
Then some straightforward computations show that $\mfk$ is in fact a simple Lie algebra and hence isomorphic to $\mfsl_3$.
\hfill \examend
\end{enumerate}
\end{exam}

\begin{prop} \label{prop:notfixed}
Let $(X,\tau) \in \GSat(A) \backslash \Sat(A)$ and $\bm \ga \in \Ga$.
Then $\mfk$ is not the fixed-point subalgebra of any $\phi \in \Aut(\mfg)$ such that~1 is a simple root of the minimal polynomial of $\phi$.
\end{prop}

\begin{proof}
We first show this in the case $(X,\tau) =$ \begin{tikzpicture}[baseline=-0.35em,line width=0.7pt,scale=75/100]
\draw[double,->] (.9,0) -- (.95,0);
\draw[line width=.5pt] (.5,0) -- (1,0);
\draw[line width=.5pt] (.5,0.05) -- (1,0.05);
\draw[line width=.5pt] (.5,-0.05) -- (1,-0.05);
\filldraw[fill=black] (.5,0) circle (.1) node[above=.25pt]{\scriptsize $1$};
\filldraw[fill=white] (1,0) circle (.1) node[above=.25pt]{\scriptsize $2$};
\end{tikzpicture}.
Suppose there exists $\phi \in \Aut(\mfg)$ such that $\mfk = \mfg^\phi$.
From $[h_2,b_1] = 3b_1$ and $[h_2,e_1]=-3e_1$ one establishes straightforwardly that $\phi(h_2) \in \mfh$ and hence that $\phi(h_2) = \tfrac{3}{2}(m-1) h_1 + m h_2$ for some $m \in \C$.
Next, from $\theta(f_2) \in \mfg_{\al_1+\al_2}$ it follows that $[h_2,b_2] = -f_2 - b_2$; hence $\phi(f_2) = m f_2 + \tfrac{1}{2}(1-m) b_2$.
Combining this with $[f_2,b_2] \in \mfn^+_X$ one obtains $m=1$.
But this means that $h_2$ and $f_2$ are also fixed points of $\phi$, contrary to assumption.
Hence such $\phi$ does not exist.

Now let $(X,\tau) \in \WSat(A)$.
Since $\mfk$ has a nonabelian nilpotent ideal by Proposition \ref{prop:k:weak:ideal}, $\mfk$ is not a reductive Lie algebra.
Hence \cite[Thm.~1]{Jac62} implies the desired conclusion.
\end{proof}

As a consequence of Proposition \ref{prop:notfixed}, $\mfk$ is not the fixed-point subalgebra of any semisimple (in particular, finite-order) automorphism of $\mfg$.\\

Finally we comment on the centre $\mfz$ of $\mfk$ for $(X,\tau) \in \WSat(A)$.
In Example \ref{ex:sp4g2} (i) we saw that it is one-dimensional if $(X,\tau)= \weakmin$\,.
Let $c \in \mfz$ and as before denote by $i$ the unique element of $I \backslash X$ such that $i = \tau(i)$ and $\al_i(\rho^\vee_X) \notin \Z$.
Proposition \ref{prop:k:weak:ideal} implies that $c=c'+c''$ with $c' \in \mfk_{\hat \imath}$, $c'' \in \mfk(i)$.
Moreover, since $c \in \mfz$ we have $[x,c'] = 0$ for all $x \in \mfk_{\hat \imath}$.
We claim that $c'=0$.
If $\mfk_{\hat \imath}$ is semisimple, we are done.
By a case-by-case analysis using Table \ref{tab:nonSatake}, note that $\mfk_{\hat \imath}$ is semisimple unless $(X,\tau) = \begin{tikzpicture}[baseline=-0.5em,line width=0.7pt,scale=75/100]
\draw[thick] (2,0) -- (3,0);
\draw[thick,dashed] (3,0) -- (4,0);
\draw[double,<-] (4.05,0) --  (4.4,0);
\filldraw[fill=white] (2.0,0) circle (.1) node[below=.2pt]{\scriptsize $1$};
\filldraw[fill=white] (2.5,0) circle (.1) node[below=.2pt]{\scriptsize $2$};
\filldraw[fill=black] (3,0) circle (.1);
\filldraw[fill=black] (4,0) circle (.1);
\filldraw[fill=black] (4.5,0) circle (.1) node[below=.2pt]{\scriptsize $n$};
\end{tikzpicture}$
with $n>2$, in which case $\mfk_{\hat \imath}$ has a one-dimensional centre spanned by $b_1$.
Since $[b_1,b_2] \ne 0$, it follows that also in this case $c'=0$.
Hence $c \in \mfk(i)$.
Since the centre of $\mfk(i)$ is $\mfk(i)_2$ we must have $\mfz \subseteq \mfk(i)_2$.
Define
\[
\mc{J}_{\rm even} := \{ \bm j \in \mc{J} \,\, | \,\, \text{the coefficient of } \al_k \text{ in } \al_{\bm j} \text{ is even for all } k \in I \backslash X \}.
\]

\begin{conj} 
If $(X,\tau) \in \WSat(A)$, a single element of $\displaystyle\bigoplus_{\bm j \in \mc{J}_{\rm even}} \C b_{\bm j} \subset \mfk(i)_2$ generates $\mfz$.
\end{conj}

\begin{rmk} 
This should be compared with the situation for Satake diagrams and the associated fixed-point subalgebras, where the centre of $\mfk = \mfg^\theta$ has dimension $|\Idiff| + |\Insf|  \in \{0,1\}$.
Casework suggests that it is generated by a linear combination of either $h_i-h_{\tau(i)}$ (for $i \in \Idiff$) or $b_i$ (for $i \in \Insf$) and at least one other standard basis element of $\mfk$.
\hfill \rmkend
\end{rmk}

%%%%%%%%%%%%%%%%%%%%%%%%%%%%%%%%%%%%%%%%%%%%%%%%%%%%%%%%%%%%%%%%

\subsection{The universal enveloping algebra $U(\mfk)$}

Let $(X,\tau) \in \GSat(A)$ and $\bm \ga \in \Ga$.
We identify $\mfk$ with its image in $U(\mfk)$ under the canonical Lie algebra embedding.
The generators of $U(\mfk)$ corresponding to $b_i$ $(i \in I \backslash X)$ can be modified by scalar terms, which is a straightforward generalization of \cite[Cor.~2.9]{Ko14}.

\begin{prop} 
For $(X,\tau) \in \GSat(A)$, $\bm \ga \in \Ga$ and $\bm \si \in \C^{I \backslash X}$, the universal enveloping algebra $U(\mfk_{\bm \ga})_{\bm \si}$ is generated by $e_i,f_i$ ($i \in X$), $h \in \mfh^\theta$ and
\eq{
\label{bis:def}
b_{i;\ga_i,\si_i} = f_i + \ga_i\,\theta(f_i) + \si_i \qq \text{for all } i \in I \backslash X.
}
\end{prop}

Again, if there is no cause for confusion, we will suppress $\bm \ga$ and $\bm \si$ from the notation.
From Conjecture \ref{conj2} we obtain the following conditional result for $U(\mfk)$, which would address the question raised in \cite[Rmk.~2.10]{Ko14}: for $(X,\tau) \in \GSat(A)$, $\bm \ga \in \Ga$ and $\bm \si \in \C^{I \backslash X}$, $U(\mfk)$ is isomorphic to the algebra with generators $h_i, e_i$ ($i \in X$), $h_i-h_{\tau(i)}$ ($i \in I^*, i \ne \tau(i)$), $b_i$ ($i \in I$) and relations \eqrefs{k:rels1}{k:Serre}.

We may view $U(\mfk)$ as a Hopf subalgebra of $U(\mfg)$ so that Lie algebra automorphisms of $\mfg$ lift to Hopf algebra automorphisms of $U(\mfg)$.
Call two Hopf subalgebras $B,B'$ of $U(\mfg)$ \emph{equivalent} if there exists $\phi \in \Aut_{\rm Hopf}(U(\mfg))$ such that $B' = \phi(B)$.
Define 
\eq{
\wt\Ga := \{ \bm \ga \in \Ga \,\, | \,\, \ga_i=1 \text{ if } i \in I^* \backslash \Idiff \}, \qq
\Si := \{ \bm \si \in \C^{I \backslash X} \,\, | \,\, \si_i = 0 \text{ if } i \in I^* \backslash  \Insf \}.
}

\begin{prop}
Let $(X,\tau) \in \GSat(A)$, $\bm \ga \in \Ga$ and $\bm \si \in \C^{I \backslash X}$.
There exist $\bm {\wt \ga} \in \wt\Ga$ and $\bm \si' \in \Si$ such that $U(\mfk_{\bm \ga})_{\bm \si}$ is equivalent to $U(\mfk_{\bm {\wt \ga}})_{\bm \si'}$.
\end{prop}

\begin{proof}
The existence of $\bm {\wt \ga}$ follows from an argument analogous to the proof of \cite[Prop.~9.2 (i)]{Ko14}.
Hence $U(\mfk_{\bm \ga})_{\bm \si}$ is equivalent to $U(\mfk_{\bm {\wt \ga}})_{\bm {\wt \si}}$ for some $\bm{\wt \si} \in \C^{I \backslash X}$.
Regarding the existence of $\bm \si' \in \Si$, note that $b_{i;\wt \ga_i} \in (\mfk_{\bm {\wt \ga}})'$ if $i \notin \Insf$, by Proposition \ref{prop:kprime}.
Hence $U(\mfk_{\bm {\wt \ga}})_{\bm {\wt \si}}$ is already generated by $e_i,f_i$ ($i \in X$), $h \in \mfh^\theta$, $b_{i;\wt \ga_i,0}$ for $i \in (I \backslash X)\backslash \Insf$ and $b_{i;\wt \ga_i,\wt s_i}$ for $i \in \Insf$.
Hence we may take $\si'_i = \wt \si_i$ if $i \in \Insf$ and $\si'_i = 0$ otherwise.
\end{proof}

%%%%%%%%%%%%%%%%%%%%%%%%%%%%%%%%%%%%%%%%%%%%%%%%%%%%%%%%%%%%%%%%%%%%%%%%%%%%%%
% Section 4
%%%%%%%%%%%%%%%%%%%%%%%%%%%%%%%%%%%%%%%%%%%%%%%%%%%%%%%%%%%%%%%%%%%%%%%%%%%%%%

\section{Quantum pair algebras and the universal K-matrix revisited} \label{sec:quantum}

Assume the $d_i$ are dyadic rationals and let $\K$ be a quadratic closure of $\C(q)$ where $q$ is an indeterminate, so that $q_i := q^{d_i} \in \K$ for all $i \in I$.
The Drinfeld-Jimbo quantum group $U_q(\mfg)$ is an associative unital algebra over $\K$ which quantizes the universal enveloping algebra $U(\mfg)$.
It is generated by $\{ E_i, F_i, K_i^{\pm 1} \}$ where $i \in I$, satisfying the relations given in e.g.~\cite[3.1.1]{Lu94}.
The Hopf algebra structure is the one defined in \cite[3.1.3, 3.1.11, 3.3]{Lu94}.
We write $U_q(\mfh)$ for the Hopf subalgebra generated by $K_i^{\pm 1}$ for $i \in I$.
We also write $U_q(\mfn^\pm)$ for the coideal subalgebras generated by the $E_i$ and $F_i$ ($i \in I$), respectively.
The algebra $U_q(\mfg)$ is $Q$-graded in terms of the root spaces $U_q(\mfg)_\al   = \{ u \in U_q(\mfg) \, | \, K_i u K_i^{-1} = q_i^{\al(h_i)} u \text{ for all } i \in I \}$.

We discuss some automorphisms of $U_q(\mfg)$.
Diagram automorphisms act (by relabelling) as Hopf algebra automorphisms on $U_q(\mfg)$.
Moreover, we have Lusztig's automorphisms $T_i$ for $i \in I$, given as $T''_{i,1}$ in \cite[37.1.3]{Lu94}, which generate an image of the braid group in $\Aut_{\rm alg}(U_q(\mfg))$ and reproduce $\Ad(s_i)$ as $q \to 1$.
They satisfy $T_i(U_q(\mfg)_\al) \subseteq U_q(\mfg)_{s_i(\al)}$ for all $\al \in Q$ and $T_i(K_j)=K_j K_i^{-a_{ij}}$ for all $j \in I$.
For $X \subseteq I$ with $w_X = s_{i_1} \cdots s_{i_\ell}$ a reduced decomposition we write $T_{w_X} = T_{i_1} \cdots T_{i_\ell}$.
If $\tau \in \Aut(A)$ stabilizes $X$ then $[\tau,T_{w_X}]=0$.
Finally, the assignments
\eq{  \label{omq:defn}
\om_q(E_i) = - K_i^{-1} F_i, \qq \om_q(F_i) = - E_i K_i, \qq \om_q(K_i^{\pm 1}) = K_i^{\mp 1} \qq \qq \text{for } i \in I
}
define $\om_q \in \Aut_{\rm alg}(U_q(\mfg))$ which is a particular quantum analogue of the Chevalley involution commuting with each $T_i$, see \cite[Lemma 7.1]{BK19}, and with $\Aut(A)$.

We now discuss the changes to definitions and statements in the papers \cite{Ko14,BK15,BK19,Ko20,DK19} needed to generalize these works to all generalized Satake diagrams.

%%%%%%%%%%%%%%%%%%%%%%%%%%%%%%%%%%%%%%%%%%%%%%%%%%%%%%%%%%%%%%%%%%%%%%%%%%%%%%

In the remainder of this section we assume $(X,\tau) \in \CD(A)$.
The quantum analogon of the map $\theta = \Ad(w_X) \tau \om$ is the map
\eq{ \label{thetaq:defn}
\theta_q = \theta_q(X,\tau) = T_{w_X} \tau\,\om_q \in \Aut_{\rm alg}(U_q(\mfg)).
}
The quantization of the fixed-point subalgebra in the formalism by \cite{Ko14} relies on the description of $\mfg^\theta$ in terms of generators given in \cite[Lemma 2.8]{Ko14}.
Our $\mfk(X,\tau)$ by definition can be quantized to a right coideal subalgebra in the same way.

\begin{defn} \label{defn:B}
Let $\bm \ga \in (\K^\times)^{I \backslash X}$ and $\bm \si \in \K^{I \backslash X}$.
The \emph{quantum pair algebra} $B = B_{\bm \ga,\bm \si}(X,\tau)$ is the coideal subalgebra generated by $U_q(\mfg_X)$, $U_q(\mfh^\theta)$ and the elements
\eq{
B_i = B_{i; \bm \ga, \bm \si} = F_i + \ga_i \theta_q(F_i K_i) K_i^{-1} + \si_i K_i^{-1} \qq \text{for all } i \in I \backslash X.
}
\end{defn}

\begin{rmk} \label{rmk:gamma:c} \mbox{}
\begin{enumerate}
\item Note that $U_q(\mfh^\theta)$ is denoted ${U^0_\Theta}'$ in \cite{Ko14}.
In \cite[Eq.~(2.4)]{Le03} and \cite[Def.~5.1]{Ko14} $\si_i$ is denoted $s_i$ (we use a different notation to avoid confusion with the simple reflections $s_i$ and the group homomorphism $s \in \wt H$).
The scalar $\ga_i$ corresponds to the scalar $d_i$ in \cite[Eq.~(2.4)]{Le03} and the scalar $c_i$ in \cite[Def.~5.1]{Ko14}.
More precisely, the Kolb-Balagovi\'{c} formalism fits in our approach upon setting 
\eq{
\ga_i = s(\al_{\tau(i)})\, c_i \qq \text{for all } i \in I \backslash X,
} 
also see \cite[Eq.~(7.7)]{BK19}.

\item Comparing with \cite[Def.~4.3]{Ko14} or \cite[Def.~5.4 and Eq.~(5.4)]{BK19}, note the absence of the factor $\Ad(s)$ from our definition of $\theta_q$.
Here $s \in \wt H$ is required to satisfy \cite[Eqs.~(5.1)-(5.2)]{BK19} so that $\theta_q$ specializes to the appropriate Lie algebra involution in the case $(X,\tau) \in \Sat(A)$, see \cite[Prop.~10.2]{Ko14}.
In our case this is unnecessary; to compare with these earlier papers there are in fact two natural choices for $s$, see \cite[Rmk.~5.2]{BK19}, one of which is satisfied for instance by $s = \chi_{(1,1,\ldots,1)}$ (see \eqref{theta_gamma} for the notation) which takes values in $\{ \pm 1 \}$.\hfill \rmkend
\end{enumerate}
\end{rmk}

Moreover, if $(X,\tau) \in \GSat(A)$ and the tuples $\bm \ga$, $\bm \si$ lie in the sets
\eq{ \label{parametersets:defn}
\begin{aligned}
\Ga_q &= \{ \bm \ga \in (\K^\times)^{I \backslash X} \, \, | \,\, \ga_i = \ga_{\tau(i)} \text{ if } i \in I^* \backslash \Idiff \}, \\
\Si_q &= \{ \bm \si \in \K^{I \backslash X} \,\, | \,\, \, \si_i = 0  \text{ if } i \in I^* \backslash \Insf \},
\end{aligned}
}
respectively, then in \cite[Sec.~5.3 and Sec.~6]{Ko14} decompositions of $B$ are obtained which culminate in the quantum analogue of \eqref{k:intersection}, namely $B \cap U_q(\mfh) = U_q(\mfh^\theta)$.
The key condition for Satake diagrams, see \eqref{Sat:defn}, is only used in \cite[Proof of Lemma 5.11, Step 1]{Ko14}, but it is clear that what is needed there is precisely the weaker condition appearing in the definition of a generalized Satake diagram, see Definition \ref{defn:GSat}.
Furthermore, in \cite[Thms.~7.4 and 7.8]{Ko14} quantum Serre relations for the $B_i$ are found for low values of $-a_{ij}$ and the results were extended in \cite[Thm.~3.7, Case 4]{BK15}; we discuss the generalization to $\GSat(A)$ in Section \ref{sec:qSerre}.
The rest of \cite{Ko14} is applicable without change in the setting of generalized Satake diagrams; in particular in the specialization $q \to 1$ one recovers $U(\mfk)$, see \cite[Sec.~10]{Ko14}.

%%%%%%%%%%%%%%%%%%%%%%%%%%%%%%%%%%%%%%%%%%%%%%%%%%%%%%%%%%%%%%%%%%%%%%%%%%%%%%

\subsection{The bar involution for the subalgebra $B$}

The bar involution $\overline{ \phantom{x} }$ of $U_q(\mfg)$ is the algebra automorphism of $U_q(\mfg)$ fixing $E_i,F_i$ and inverting $K_i^{\pm 1}$ and $q$; it plays a crucial role in Lusztig's construction of the quasi R-matrix, see \cite{Lu94}.
In order to have a natural modification of Lusztig's theory of bar involutions and quasi R-matrices to the setting of quantum symmetric pair algebras, the paper \cite{BK15} deals with the existence of an analogous map of $B$.
This follows earlier work by \cite{ES13} and \cite{BW18a} in the case of certain quantum symmetric pairs of type AIII.
More precisely, for suitable parameters $\bm \ga$, there exists an involutive algebra automorphism $\overline{ \phantom{x} }^B : B \to B$ which coincides with $\overline{ \phantom{x} }$ on $U_q(\mfg_X)U_q(\mfh^\theta)$ and satisfies $\overline{ B_i }^B = B_i$ for $i \in I \backslash X$, see \cite[Thm.~3.11, Cor.~3.13, Rmk.~3.15]{BK15}.
The defining condition of Satake diagrams is not used explicitly in \cite{BK15} but casework is used in the results \cite[Props.~2.3 and 3.8]{BK15} which needs to be extended to the new diagrams in Table \ref{tab:nonSatake}, which we will now explain.
We also note that the result \cite[Prop.~2.3]{BK15} was generalized in \cite[Thm.~4.1]{BW18c} to the Kac-Moody setting, but this did not explicitly include the cases where $B$ is defined in terms of $(X,\tau) \in \GSat(A) \backslash \Sat(A)$.
Here we provide an elementary proof for $\mfg$ of finite type which works for all compatible decorations.
 
Let $\si$ be the unique algebra anti-automorphism of $U_q(\mfg)$ which fixes $E_i$ and $F_i$ and inverts $K_i$.
For $i \in I$, denote by $r_i$ Lusztig's \emph{right skew derivation}, see \cite[1.2.13]{Lu94}; it is the unique linear map on $U_q(\mfn^+)$ such that for all $x,y \in U_q(\mfn^+)$ with $y \in U_q(\mfg)_\mu$ ($\mu \in Q^+$) we have
\eq{ \label{ri:def}
r_i(x y) = q_i^{\mu(h_i)} r_i(x) y + x r_i(y).
}
We denote $[x,y]_p:=xy-pyx$ for $x,y \in U_q(\mfg)$ and $p \in \K$; note that $\si([x,y]_p) = [\si(y),\si(x)]_p$.
The definition of $T_j$ implies that $T_j(E_i)=E_i$ if $a_{ji}=0$ and $T_j(E_i) = [E_j,E_i]_{q_j^{-1}}$ if $a_{ji}=-1$.
We start with a lemma that simplifies the proof drastically.
Call a connected component of $X$ \emph{simple} if it is of the form $\{j\}$ for some $j \in I$ such that $a_{ij}=a_{ji} \in \{0,-1\}$ for all $i\in I \backslash X$.

\begin{lemma} \label{lem:simplecomponent}
Let $i \in I \backslash X$ and suppose $\emptyset \ne X(i) = X_1 \cup \cdots \cup X_\ell$ is a decomposition into connected components.
If $i \ne \tau(i)$ then $\ell \le 1$ and if $i = \tau(i)$, all $X_s$ are simple except at most one.
Denote by $Y$ the non-simple connected component of $X(i)$ if present and otherwise let $Y$ be any simple connected component.
If $(r_i T_{w_Y})(E_i)$ is fixed by $\si \tau$ then so is $(r_i T_{w_X})(E_i)$.
\end{lemma}

\begin{proof}
The first part of the Lemma follows from the classification of Satake diagrams in \cite{Ar62} and an inspection of Table \ref{tab:nonSatake}.
Since adding simple components does not change the statement, it is true for all compatible decorations.

The second part is proven by induction with respect to the number of simple components.
If there are none, then $X(i) = Y$ and the statement is true.
Otherwise, by the induction hypothesis we may suppose $X(i) = X' \cup \{j\}$ where $(\si \tau)(r_i T_{w_{X'}})(E_i)) = (r_i T_{w_{X'}})(E_i)$, $\{ j \}$ is a simple component of $X$ and $a_{jk} = 0$ for all $k \in X'$.
Hence $T_{X(i)} = T_{w_{X'}} T_j$ so that $T_{w_X}(E_i) = T_{X(i)}(E_i) = T_{w_{X'}}([E_j,E_i]_{q_j^{-1}}) = [E_j,T_{w_{X'}}(E_i)]_{q_j^{-1}}$.
By \eqref{ri:def} we have $(r_i T_{w_X})(E_i) = [E_j,(r_i T_{w_{X'}})(E_i)]_{q_j^{-2}}$.
Since $(r_i T_{w_{X'}})(E_i)$ lies in $U_q(\mfg_{X'})$ it commutes with $E_j$.
Hence
\eq{ 
(r_i T_{w_X})(E_i) = (1-q_j^{-2}) E_j (r_i T_{w_{X'}})(E_i) = (1-q_j^{-2}) (r_i T_{w_{X'}})(E_i) E_j.
}
Since $\tau(j)=j$, applying $\si \tau$ yields the desired result.
\end{proof}

\begin{prop} \label{prop:nui1}
For all $i \in I \backslash X$, $(r_i T_{w_X})(E_i)$ is fixed by $\si \tau$.
\end{prop}

\begin{proof}
The proof is essentially casework, but first we make some observations.

\begin{enumerate}
\item 
Since $T_{w_X}(E_i) = T_{X(i)}(E_i)$ we may assume that $\{ i, \tau(i) \}$ is the only $\tau$-orbit outside $X$.

\item
We may assume $X$ is nonempty as otherwise $(r_i T_{w_X})(E_i) = 1$.

\item 
By Lemma \ref{lem:simplecomponent} it suffices to prove the statement in the case that $X$ is connected.

\item 
If $|X|=1$, we write $X=\{j\}$ with $\tau(j)=j$.
Then $T_{w_X}(E_i) = T_j(E_i) \in U_q(\mfg)_{s_j(\al_i)} \cap U_q(\mfn^+)$.
Hence $(r_i T_{w_X})(E_i) \in U_q(\mfg)_{s_j(\al_i)-\al_i} \cap U_q(\mfn^+) = \K E_j^{-a_{ji}}$ so it is fixed by $\si \tau$.

\item In \cite[Proof of Prop.~2.3]{BK15} the statement was proved for all Satake diagrams.
\end{enumerate}

Hence it suffices to prove the statement for those diagrams in Table \ref{tab:nonSatake} where the node $i$ is the only node outside $X$, $X$ is connected and $|X|>1$.
The only infinite family satisfying this condition is given by 
\begin{tikzpicture}[baseline=-0.2em,line width=0.7pt,scale=75/100]
\draw[thick] (2.5,0) -- (3,0);
\draw[thick,dashed] (3,0) -- (4,0);
\draw[double,<-] (4.05,0) --  (4.4,0);
\filldraw[fill=white] (2.5,0) circle (.1);
\filldraw[fill=black] (3,0) circle (.1);
\filldraw[fill=black] (4,0) circle (.1);
\filldraw[fill=black] (4.5,0) circle (.1);
\end{tikzpicture}.
In this case the proof is identical to the proof for the type BII case in \cite[Prop.~2.3]{BK15}.
The exceptional diagrams satisfying this condition are
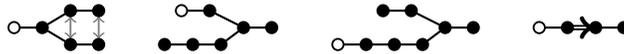
\[
\begin{tikzpicture}[baseline=-0.35em,line width=0.7pt,scale=75/100]
\draw[thick] (1,.3) -- (.5,.3) -- (0,0) -- (.5,-.3) -- (1,-.3);
\draw[thick] (0,0) -- (-.5,0);
\draw[<->,gray] (1,.2) -- (1,-.2);
\draw[<->,gray] (.5,.2) -- (.5,-.2);
\filldraw[fill=black] (1,.3) circle (.1);
\filldraw[fill=black] (1,-.3) circle (.1);
\filldraw[fill=black] (.5,.3) circle (.1);
\filldraw[fill=black] (.5,-.3) circle (.1);
\filldraw[fill=black] (0,0) circle (.1);
\filldraw[fill=white] (-.5,0) circle (.1);
\end{tikzpicture}
\qq
\begin{tikzpicture}[baseline=-0.35em,line width=0.7pt,scale=75/100]
\draw[thick] (-1.1,.3) -- (-.6,.3) -- (0,0) -- (-.4,-.3) -- (-1.4,-.3);
\draw[thick] (0,0) -- (.5,0);
\filldraw[fill=white] (-1.1,.3) circle (.1);
\filldraw[fill=black] (-.6,.3) circle (.1);
\filldraw[fill=black] (0,0) circle (.1);
\filldraw[fill=black] (-.4,-.3) circle (.1);
\filldraw[fill=black] (-.9,-.3) circle (.1);
\filldraw[fill=black] (-1.4,-.3) circle (.1);
\filldraw[fill=black] (.5,0) circle (.1);
\end{tikzpicture}
\qq
\begin{tikzpicture}[baseline=-0.35em,line width=0.7pt,scale=75/100]
\draw[thick] (-1.9,-.3) -- (-.4,-.3) -- (0,0) -- (-.6,.3) -- (-1.1,.3);
\draw[thick] (0,0) -- (.5,0);
\filldraw[fill=black] (-1.1,.3) circle (.1);
\filldraw[fill=black] (.5,0) circle (.1);
\filldraw[fill=black] (-.6,.3) circle (.1);
\filldraw[fill=black] (0,0) circle (.1);
\filldraw[fill=black] (-.4,-.3) circle (.1);
\filldraw[fill=black] (-.9,-.3) circle (.1);
\filldraw[fill=black] (-1.4,-.3) circle (.1);
\filldraw[fill=white] (-1.9,-.3) circle (.1);
\end{tikzpicture}
\qq
\begin{tikzpicture}[baseline=-0.35em,line width=0.7pt,scale=75/100]
\draw[thick] (.5,0) -- (1,0);
\draw[double,->] (1,0) -- (1.45,0);
\draw[thick] (1.5,0) -- (2,0);
\filldraw[fill=white] (.5,0) circle (.1);
\filldraw[fill=black] (1,0) circle (.1);
\filldraw[fill=black] (1.5,0) circle (.1);
\filldraw[fill=black] (2,0) circle (.1);
\end{tikzpicture}
\]
We give here the proof for the last case which is very much in the spirit of \cite[Proof of Prop.~2.3]{BK15}; the proofs for the other three cases are similar and are left to the reader.

We label the nodes as \begin{tikzpicture}[baseline=-0.2em,line width=0.7pt,scale=75/100]
\draw[thick] (.5,0) -- (1,0);
\draw[double,->] (1,0) -- (1.45,0);
\draw[thick] (1.5,0) -- (2,0);
\filldraw[fill=white] (.5,0) circle (.1) node[above]{\scriptsize $1$};
\filldraw[fill=black] (1,0) circle (.1) node[above]{\scriptsize $2$};
\filldraw[fill=black] (1.5,0) circle (.1) node[above]{\scriptsize $3$};
\filldraw[fill=black] (2,0) circle (.1) node[above]{\scriptsize $4$};
\end{tikzpicture}
and assume $d_1=d_2=2$ and $d_3=d_4=1$ for convenience.
Note that $\tau = \id$.
The reduced decompositions $w_X =( s_2 s_3 s_2 s_4 s_3 s_2 )(s_4 s_3 s_4) = (s_4 s_3 s_4) ( s_2 s_3 s_2 s_4 s_3 s_2 )$ yield 
\eq{ \label{TX:E1}
T_{w_X}(E_1) = (T_2 T_3 T_2 T_4 T_3 T_2)(E_1) = (T_4 T_3 T_4 T_{w_X})(E_1).
}
From the first expression we readily obtain
\eq{
T_{w_X}(E_1) = [(T_2 T_3 T_2 T_4 T_3)(E_2),[(T_2 T_3)(E_2),[E_2,E_1]_{q^{-2}}]_{q^{-2}}]_{q^{-2}}.
}
Now note that $(s_3 s_4 s_2 s_3 s_2 s_4 s_3)(\al_2) = (s_3s_2s_3)(\al_2) = \al_2$ and $s_3 s_4 s_2 s_3 s_2 s_4 s_3$ and $s_3 s_2 s_3$ are reduced elements in $W$.
Appealing to \cite[Prop.~8.20]{Jan96} we arrive at 
\eq{
T_{w_X}(E_1) = [ (T_4^{-1} T_3^{-1})(E_2) , [T_3^{-1}(E_2),[E_2,E_1]_{q^{-2}}]_{q^{-2}}]_{q^{-2}}
}
so that \eqref{ri:def} implies
\eq{
(r_1 T_{w_X})(E_1) = (1-q^{-4})  [ (T_4^{-1} T_3^{-1})(E_2) , [T_3^{-1}(E_2), E_2]_{q^{-4}}]_{q^{-4}}.
}
Applying $\si \tau$ and using $T_i \si = \si T_i^{-1}$ (see e.g.~\cite[37.2.4]{Lu94}) we obtain
\eq{ \label{sigmatau}
\begin{aligned}
&(\si \tau)((r_1 T_{w_X})(E_1)) = (1-q^{-4})  [ [E_2,T_3(E_2)]_{q^{-4}},(T_4 T_3)(E_2)]_{q^{-4}} \\
& \qq = (1-q^{-4}) \left( q^{-4} [T_3(E_2),T_4([T_3(E_2),E_2])] + [E_2,[T_3(E_2),(T_4 T_3)(E_2)]_{q^{-4}}]_{q^{-4}} \right).
\end{aligned}
}
By the q-Serre relation $E_2^2 E_3 - (q^2+q^{-2}) E_2 E_3 E_2 + E_3 E_2^2 = 0$ we have
\eq{
\begin{aligned}
[T_3(E_2),E_2] &= \frac{1}{q+q^{-1}}[ E_3^2 E_2 - (1+q^{-2}) E_3 E_2 E_3 + q^{-2} E_2 E_3^2,E_2]  \\
&= \frac{q^2-1}{q+q^{-1}} \left( E_3 E_2 E_3 E_2  - q^{-2} E_3 E_2^2 E_3 - q^{-2} E_2 E_3^3 E_2 + q^{-4} E_2 E_3 E_2 E_3\right) \\
&= (q^2-1)[E_3,E_2]_{q^{-2}}^2 = (q^2-1) \si\big( [E_2,E_3]_{q^{-2}}^2 \big) \\
&=  (q^2-1) ( \si T_2)(E_3^2) \hspace{1mm} =  (q^2-1) T_2^{-1}(E_3^2 )
\end{aligned}
}
so that for the first term of \eqref{sigmatau} we have
\eq{ 
[T_3(E_2),T_4([T_3(E_2),E_2])] = (q^2-1) T_2^{-1} \left( [(T_2 T_3)(E_2),T_4(E_3^2)]\right).
}
The reduced elements $s_3 s_2 s_3$ and $s_3 s_4$ map $\al_2$ to itself and $\al_3$ to $\al_4$, respectively, so that $(T_2 T_3)(E_2) = T_3^{-1}(E_2)$ and $T_4(E_3) = T_3^{-1}(E_4)$ by \cite[Prop.~8.20]{Jan96}.
Hence
\eq{ 
[T_3(E_2),T_4([T_3(E_2),E_2])] = (q^2-1) (T_2^{-1}T_3^{-1}) \left( [E_2,E_4^2]\right) = 0
}
where we have used the q-Serre relation $E_2 E_4 - E_4 E_2 = 0$.
As a consequence, \eqref{sigmatau} yields
\eq{
\begin{aligned}
(\si \tau)((r_1 T_{w_X})(E_1)) &= (1-q^{-4})  [ E_2,[T_3(E_2),(T_4 T_3)(E_2)]_{q^{-4}}]_{q^{-4}} \\
&=  (1-q^{-4})  (T_4 T_3 T_4) \left([ (T_4^{-1} T_3^{-1})(E_2) , [T_3^{-1}(E_2), E_2]_{q^{-4}}]_{q^{-4}} \right) \\
&= (T_4 T_3 T_4 r_1T_{w_X})(E_1)
\end{aligned}
}
where we have used $T_3 T_4 T_3 = T_4 T_3 T_4$.
Because $r_i$ and $T_j$ commute if $a_{ij}=0$ we have $(\si \tau)((r_1 T_{w_X})(E_1))  = (r_1 T_4 T_3 T_4 T_{w_X})(E_1)$ and by virtue of \eqref{TX:E1} the proof is complete.
\end{proof}

%%%%%%%%%%%%%%%%%%%%%%%%%%%%%%%%%%%%%%%%%%%%%%%%%%%%%%%%%%%%%%%%%%%%%%%%%%

\subsection{Quantum Serre relations for the $B_i$} \label{sec:qSerre}

We now introduce some notation in order to discuss q-Serre relations for the $B_i$.
For $x,y \in U_q(\mfg)$ and $i,j \in I$ such that $i \ne j$ recall the shorthand $M_{ij}=1-a_{ij}$ and define
\eq{
F_{ij}(x,y) = \sum_{n=0}^{M_{ij}} (-1)^n {M_{ij} \brack n}_{q_i} x^{M_{ij}-n} y x^n.
}
Here ${M_{ij} \brack n}_{q_i}$ is the $q_i$-deformed binomial coefficient defined in \cite[1.3.3]{Lu94}.
Note that the q-Serre relations for $U_q(\mfg)$ are the relations $F_{ij}(E_i,E_j) = F_{ij}(F_i,F_j) = 0$ for all $i,j \in I$ with $i \ne j$.
We denote $B_i = F_i$ if $i \in X$ and $U_q(\mfn^+_X) = \langle E_i \, | \, i \in X \rangle$.
For $i \in I \backslash X$ we write
\eq{ \label{Yi:def}
Y_i = (r_{\tau(i)} \theta_q)(F_i K_i) K_i^{-1} K_{\tau(i)} = -(r_{\tau(i)} T_{w_X})(E_{\tau(i)}) K_i^{-1} K_{\tau(i)}
}
and we write $[n]_{q_i} = \frac{q_i^n-q_i^{-n}}{q_i-q_i^{-1}}$ for $n \in \Z$, $i \in I$.

According to \cite[Thm.~3.7 (Case 4) and Thm.~3.9]{BK15}, for all $i,j \in I$ such that $i \ne j$ the elements $F_{ij}(B_i,B_j)$ are $U_q(\mfn^+_X)U_q(\mfh^\theta)$-linear combinations of products $B_{i_1} \cdots B_{i_\ell}$ with $i_1, \ldots, i_\ell \in I$ satisfying $\al_{i_1} + \ldots + \al_{i_\ell} < \la_{ij}$.
Moreover, these expressions vanish if $\tau(i) \notin \{i,j\}$ as a consequence of \cite[Thm.~7.3]{Ko14} and if $i \in X$, see \cite[Eq.~(5.20)]{Ko14}.
If $\tau(i)=j \in I \backslash X$ they are given by \cite[Thm.~3.6]{BK15} for all values of $a_{ij}$.
Now suppose $\tau(i)=i \in I \backslash X$.
Then the expressions are given in \cite[Thm.~3.7]{BK15} if $j \in I \backslash (X \cup \{ i \})$ and $a_{ij} \in \{ 0,-1,-2,-3\}$ and in \cite[Thm.~3.8]{BK15} if $j \in X$ and $a_{ij} \in \{ 0,-1,-2\}$.

These theorems cover all generalized Satake diagrams in Table \ref{tab:nonSatake} except \begin{tikzpicture}[baseline=-0.25em,line width=0.7pt,scale=75/100]
\draw[double,->] (.9,0) -- (.95,0);
\draw[line width=.5pt] (.5,0) -- (1,0);
\draw[line width=.5pt] (.5,0.05) -- (1,0.05);
\draw[line width=.5pt] (.5,-0.05) -- (1,-0.05);
\filldraw[fill=black] (.5,0) circle (.1) node[left]{\scriptsize 1};
\filldraw[fill=white] (1,0) circle (.1) node[right]{\scriptsize 2};
\end{tikzpicture}
which we discuss now; without loss of generality we may assume $d_1=3$ and $d_2=1$.
Note that $Y_2 = -r_2(T_1(E_2)) = -q^{-3} (q^3 - q^{-3}) E_1$.
In this case we have, by a direct computation,
\eq{ \label{qSerre:nonweak}
\begin{aligned}
F_{21}(B_2,B_1) &= \left( ([3]_{q}^2 + 1)(B_2^2 B_1 + B_1 B_2^2) -  [4]_{q} ([2]_{q}^2 + 1) B_2 B_1 B_2 \right) q \ga_2 Y_2 + \\
& \qu  - [3]_{q}^2 B_1 (q \ga_2 Y_2)^2 + [2]_{q} B_2^2 \big(r_1(q \ga_2 Y_2) q^6 K_1 + (\si r_1 \si)(q \ga_2 Y_2) q^{-6} K_1^{-1} \big) + \\
& \qu - \frac{[3]_{q}}{(q-q^{-1})(q^3-q^{-3})} q \ga_2 Y_2 \big( r_1(q \ga_2 Y_2) q^9 K_1 + (\si r_1 \si)(q \ga_2 Y_2) q^{-9} K_1^{-1} \big).
\end{aligned}
}

\begin{rmk}
In the formula \cite[Thm.~3.7,Eq.~(3.21)]{BK15} for $a_{ij}=-3$, which deals with the case $j \in I \backslash X$, the sign of the terms quadratic in $B_i$ is incorrect.
This does not affect the later result \cite[Thm.~3.11]{BK15}.
Up to this sign issue, the first terms of \eqref{qSerre:nonweak} directly match \cite[Eq.~(3.21)]{BK15}, so that also in the case $a_{ij}=-3$ we expect a unified expression in the style of \cite[Thm.~3.9]{BK15} in the general Kac-Moody setting.\hfill \rmkend
\end{rmk}

We now review the arguments from \cite{BK15} in terms of $\ga_i$ and $Y_i$ to show that the bar involution exists for general $(X,\tau) \in \GSat(A)$ with $A$ of finite type.
First of all, note that the proof of \cite[Prop.~3.5]{BK15} does not require the specific choice of $s \in \wt H$ given in \cite[Eq.~(3.2)]{BK15} but any $s$ satisfying the weaker constraint \cite[Eqs.~(5.1-5.2)]{BK19}, also see \cite[Rmk.~5.2]{BK19}.
With that qualification, the arguments of the proof of \cite[Prop.~3.5]{BK15} imply that 
\eq{ \label{bar:Y}
\overline{Y_i} = q_i^{(\al_i-w_X(\al_i)-2\rho_X)(h_i)} \zeta(\al_i) Y_{\tau(i)}
}
where $\rho_X \in (\mfh_X)^*$ is the Weyl vector of $\mfg_X$; this specializes to \cite[Prop.~3.5]{BK15} if we set $\ga_i = s(\al_{\tau(i)})c_i$.
Note that Proposition \ref{prop:nui1} is crucial in this context; in the notation of \cite{BK15,BK19}, we have $\nu_i = 1$ for all $i \in I \backslash X$, for all $(X,\tau) \in \GSat(A)$ of finite type.
The proof of \cite[Thm.~3.11]{BK15}, which requires \eqref{qSerre:nonweak} in the special case \nonweak, implies that the existence of the bar involution for $B$ is equivalent to
\eq{ \label{bar:gammaY}
\overline{\ga_i Y_i} = q_i^{\al_{\tau(i)}(h_i)} \ga_{\tau(i)} Y_{\tau(i)}
}
for all $i \in I \backslash X$, which specializes to \cite[Eq.~(3.28)]{BK15} if $\ga_i = s(\al_{\tau(i)})c_i$.
Combining \eqrefs{bar:Y}{bar:gammaY}, the existence of the bar involution for $B$ is equivalent to
\eq{ \label{gamma:bar}
\ga_{\tau(i)} = \zeta(\al_i) q_i^{(\theta(\al_i)-2\rho_X)(h_i)} \overline{\ga_i}
}
where we have used the analogue of \eqref{theta:h} for roots, which can be obtained in the same way.

\begin{rmk}
Suppose $i \in I \backslash X$ is such that $\tau(i)=i$ and $\al_i(\rho^\vee_X) \notin \Z$.
Then \eqref{gamma:bar} simplifies to $\ga_i = -q_i^{(\theta(\al_i)-2\rho_X)(h_i)} \overline{\ga_i}$ so that $\ga_i \to 0$ as $q \to 1$.
For $(X,\tau) \in \WSat(A)$ this is compatible with Proposition \ref{prop:k:weak:ideal} (iv).\hfill \rmkend
\end{rmk}

%%%%%%%%%%%%%%%%%%%%%%%%%%%%%%%%%%%%%%%%%%%%%%%%%%%%%%%%%%%%%%%%%%%%%%%%%%%%%%

\subsection{The universal K-matrix}

Building upon the work in \cite{BK15} on the bar involution, in \cite{BK19} the universal K-matrix $\mc K = \mc K(X,\tau)$ for $B = B(X,\tau)$ is constructed for $(X,\tau) \in \Sat(A)$ and the relations \eqref{Uintw} and \eqref{copr} are derived.
Again we comment on those places in this text with a nontrivial generalization to the setting of quantum pair algebras defined in terms of generalized Satake diagrams.
In addition to \eqref{gamma:bar} we also assume 
\eq{ \label{sigma:bar}
\overline{\si_i} = \si_i,
}
see \cite[Eq.~(5.16)]{BK19}.
In \cite[Proof of Lemma 6.4]{BK19} the defining condition of Satake diagrams is used, but only the defining condition of generalized Satake diagrams is needed.
In \cite[Rmk.~7.2]{BK19} a case-by-case analysis is employed based on the list of Satake diagrams in \cite{Ar62}.
Denoting by $\tau_0$ the diagram automorphism corresponding to the longest element of the Weyl group of $\mfg$, it is derived that $\tau_0$ preserves $X$ and commutes with $\tau$; furthermore one can choose $\bm \ga \in \Gamma$ and $\bm \si \in \Si$ such that $\ga_{\tau(i)} = \ga_{\tau_0(i)}$ and $\si_{\tau(i)} = \si_{\tau_0(i)}$ for all $i \in I \backslash X$ and the above transformation properties with respect to the bar involution hold.
Extending this analysis to Table \ref{tab:nonSatake}, one checks that the remark is valid for all generalized Satake diagrams.

In \cite{Ko20} it is shown that $\mc{K}$ satisfies the axiom \eqref{inBU} for a universal K-matrix and the centre of $B$ is described in terms of $\mc{K}$, without using the defining condition of Satake diagrams or a case-by-case analysis; it follows that the results remain valid for $(X,\tau) \in \GSat(A)$.

This is also the case for \cite[Section 2]{DK19} which entails an analysis of the restricted Weyl group and restricted root system following \cite{Lu76} in order to establish a factorization of the quasi K-matrix (subject to a condition on Satake diagrams of restricted rank 2).
In particular, from the statement $\tau_0(X)=X$ for any $(X,\tau) \in \Sat(A)$ it is inferred that $\tau_{0,X[i]}(X)=X$ for all $i \in I^*$ and hence $\{ w_X w_{X[i]} \, | \, i \in I^*\}$ is a Coxeter system for the group it generates.
For all generalized Satake diagrams these conclusions follow directly from Theorem \ref{thrm:Heck}.

%%%%%%%%%%%%%%%%%%%%%%%%%%%%%%%%%%%%%%%%%%%%%%%%%%%%%%%%%%%%%%%%%%%%%%%%%%%%%%

\paragraph{\bf Acknowledgments}

The authors thank A.~Appel, M.~Balagovi\'c, S.~Kolb, J.~Stokman, W.~Wang and an anonymous referee for useful comments and discussions.
V.R.~was supported by the European Social Fund, grant number 09.3.3-LMT-K-712-02-0017.
B.V.~was supported by the Engineering and Physical Sciences Research Council, grant numbers EP/N023919/1 and EP/R009465/1.
The authors gratefully acknowledge the financial support.

%%%%%%%%%%%%%%%%%%%%%%%%%%%%%%%%%%%%%%%%%%%%%%%%%%%%%%%%%%%%%%%%%%
% Appendices
%%%%%%%%%%%%%%%%%%%%%%%%%%%%%%%%%%%%%%%%%%%%%%%%%%%%%%%%%%%%%%%%%%

\appendix 

\section{Deriving Serre relations for $\mfk$} \label{appendix}

The following three technical lemmas are used to derive the key equation \eqref{k:Serre}.
It is convenient to introduce the notation $Q^+_X := Q^+ \cap Q_X = \sum_{i \in X} \Z_{\ge 0} \al_i$.

\begin{lemma} \label{lem:biij:a}
Fix $(X,\tau) \in \CD(A)$, $\bm \ga \in \C^{I \backslash X}$, $i \in X$ and $j \in I$.
For all $m \in \Z_{\ge 1}$ we have
\eq{
\ad(b_i)^m(b_j) = 
\begin{cases} 
\ad(f_i)^m(f_j) + \ga_j \,\theta\left( \ad(f_i)^m(f_j) \right) & \text{if } j \in I \backslash X, \\
\ad(f_i)^m(f_j) & \text{if } j \in X.
\end{cases}
}
\end{lemma}

\begin{proof}
Because $\theta$ is a Lie algebra automorphism this follows immediately from \eqref{theta:basic} 
\end{proof}

\begin{lemma} \label{lem:biij:b}
Fix $(X,\tau) \in \CD(A)$, $\bm \ga \in \C^{I \backslash X}$, $i \in I \backslash X$ and $j \in X$.
For all $m \in \Z_{\ge 1}$ we have
\eq{
\ad(b_i)^m(b_j) = \ad(f_i)^m(f_j) + \ga_i^m\, \theta\left( \ad(f_i)^m(f_j) \right) + L_m
}
where
\eq{
L_m = 
\begin{cases} 
(1+\zeta(\al_i)) \,\ga_i\, [\theta(f_i),[f_i,f_j]] \in \mfn_X^+ & \text{if } \tau(i)=i, \, w_X(\al_i)-\al_i-\al_j \in \Phi^+, \, m=2,  \\
-\ga_i\, (2h_i-a_{ij}h_j) & \text{if } \tau(i)=i, \, w_X(\al_i)=\al_i+\al_j, \, m=2, \\
-3\,(2+a_{ij}) \,\ga_i\, (f_i - \theta(f_i)) & \text{if } \tau(i)=i, \, w_X(\al_i)=\al_i+\al_j, \, m=3, \\
-6\,a_{ij}(2+a_{ij}) \,\ga_i^2\, e_j & \text{if } \tau(i)=i, \, w_X(\al_i)=\al_i+\al_j, \, m=4, \\
0 & \text{otherwise}.
\end{cases}
\hspace{-2mm} 
}
\end{lemma}

\begin{proof}
By induction with respect to $m$.
For $m=1$, \eqref{theta:basic} implies
\eq{
\ad(b_i)^1(b_j) = [f_i + \ga_i\,\theta(f_i),f_j] = \ad(f_i)^1(f_j) + \ga_i^1\, \theta\left( \ad(f_i)(f_j) \right) + L_1
}
with $L_1=0$ as required.
Now assume $m \in \Z_{>1}$ and suppose the statement holds for all smaller values.
Then, by virtue of the induction hypothesis, the fact that $\theta$ is a Lie algebra automorphism and \eqref{theta:basic}, we find
\eq{
\begin{aligned}
\ad(b_i)^m(b_j) &= \left[b_i,\ad(b_i)^{m-1}(b_j)\right] \\
&= \left[f_i + \ga_i\, \theta(f_i), \ad(f_i)^{m-1}(f_j) + \ga_i^{m-1} \theta\left( \ad(f_i)^{m-1}(f_j) \right) + L_{m-1}\right] \\
&= \ad(f_i)^m(f_j) + \ga_i^m \theta\left( \ad(f_i)^m(f_j)  \right) \\
& \qu + \ga_i \left[\theta(f_i),\ad(f_i)^{m-1}(f_j) \right] + \ga_i^{m-1} \left[f_i,\theta\left( \ad(f_i)^{m-1}(f_j) \right) \right] +  \left[b_i,L_{m-1}\right].
\hspace{-2mm}
\end{aligned}
}
Using \eqref{theta:basic} we have $\theta^2(f_i)=\zeta(\al_i) f_i$ so that
\eq{ 
\label{recursion:b}
L_m = \ga_i \left[\theta(f_i), \ad(f_i)^{m-1}(f_j) \right] + \zeta(\al_i) \ga_i^{m-1} \theta\left( \left[\theta(f_i), \ad(f_i)^{m-1}(f_j) \right] \right) + [b_i, L_{m-1}].
}
Suppose that $\left[\theta(f_i), \ad(f_i)^{m-1}(f_j) \right] \ne 0$.
Then $w_X(\al_{\tau(i)})-(m-1)\,\al_i-\al_j \in \Phi \cup \{0\}$.
Now $\Phi = \Phi^+ \cup \Phi^-$ implies that $\tau(i)=i$ and $j \in X(i)$.

First we consider the case $w_X(\al_i)-(m-1)\,\al_i-\al_j  \in \Phi^+$.
Then $m=2$ and since $w_X(\al_i) - \al_i - \al_j \in Q^+_X$ it follows that $[\theta(f_i),[f_i,f_j]] \in \mfn^+_X$.
The claimed expression for $L_2$ follows immediately from \eqref{recursion:b} and those for $L_m$ with $m>2$ from \eqref{k:rels1}.

If $w_X(\al_i)-(m-1)\,\al_i-\al_j \in \Phi^- \cup \{0\}$ then $w_X(\al_i) \le (m-1)\, \al_i + \al_j$.
Hence $X(i)=\{j\}$ and we obtain 
\eq{
w_X(\al_i) - (m-1)\, \al_i - \al_j = (2-m)\,\al_i - (1+a_{ji})\,\al_j.
}
From $\Phi = \Phi^+ \cup \Phi^-$ it follows that $a_{ji} = -1$.
Now $\Z \al_i \cap \Phi = \{ \pm \al_i \}$ implies that $m \in \{ 2,3 \}$.
Straightforwardly we compute 
\eq{
\left[\theta(f_i), \ad(f_i)(f_j) \right] = a_{ij} h_j - h_i, \qq \left[\theta(f_i), \ad(f_i)^2(f_j) \right] = -2(1+a_{ij}) f_i,
} 
from which the claimed expressions for $L_m$ readily follow.
\end{proof}

Recall the integer $p_{ij}^{(r,m)}$ defined by \eqref{pij:def}.

\begin{lemma} \label{lem:biij:c}
Fix $(X,\tau) \in \CD(A)$, $\bm \ga \in \C^{I \backslash X}$ and $i,j \in I \backslash X$ such that $i \ne j$.
Recall the integer $p_{ij}^{(r,m)}$ defined by \eqref{pij:def}.
For all $m \in \Z_{\ge 0}$ we have
\eq{ \label{biij:c}
\ad(b_i)^m(b_j) = \ad(f_i)^m(f_j) + \ga_i^m \ga_j \theta\left( \ad(f_i)^m(f_j) \right) + L_m
}
where
\eq{
L_m = 
\begin{cases} 
\left( \ga_i + \zeta(\al_i) \ga_j \right) [\theta(f_i),f_{j}] \in \mfn^+_X & \text{if } \tau(i)=j, \, w_X(\al_i)-\al_i \in \Phi^+, \, m=1, \\
\ga_j h_i - \ga_i h_j & \text{if } \tau(i)=j, \, w_X(\al_i) = \al_i, \, m=1, \\
2\left( (\ga_j - a_{ij} \ga_i)  f_i - \ga_i (\ga_i -a_{ij} \ga_j)  e_j \right) & \text{if } \tau(i)=j, \, w_X(\al_i) = \al_i, \, m=2, \\
\displaystyle\sum_{r=1}^{\lfloor m/2 \rfloor} p_{ij}^{(r,m)} \ga_i^r \ad(b_i)^{m-2r}(b_j) & \text{if } \tau(i)=i, \, w_X(\al_i) = \al_i, \\
0 & \text{otherwise}.
\end{cases}
}
\end{lemma}

\begin{proof}
As before we apply induction with respect to $m$.
For $m=0$ we have
\eq{
\ad(b_i)^0(b_j) = b_j = f_j + \ga_j\, \theta(f_j) = \ad(f_i)^0(f_j) + \ga_i^0 \ga_j\, \theta\left( \ad(f_i)^0(f_j) \right) + L_0
}
with $L_0=0$ as required.
Now assume $m \in \Z_{>0}$ and suppose the statement holds for all smaller values.
Then, by the induction hypothesis,
\eq{
\begin{aligned} 
\ad(b_i)^{m}(b_j) &= [b_i,\ad(b_i)^{m-1}(b_j)] \\
&= [f_i + \ga_i \theta(f_i), \ad(f_i)^{m-1}(f_j) + \ga_i^{m-1} \ga_j \theta\left( \ad(f_i)^{m-1}(f_j) \right) + L_{m-1}].
\end{aligned}
}
Since $\theta$ is a Lie algebra automorphism, rearranging terms yields \eqref{biij:c} with $L_m \in \mfg$ defined by 
\eq{
L_m = \ga_i \left[\theta(f_i), \ad(f_i)^{m-1}(f_j) \right] + \ga_i^{m-1} \ga_j \left[f_i,\theta\left( \ad(f_i)^{m-1}(f_j) \right) \right] +  [b_i, L_{m-1}].
}
Using \eqref{theta:basic} we obtain
\eq{ 
\label{recursion:c}
L_m = \ga_i \left[\theta(f_i), \ad(f_i)^{m-1}(f_j) \right] + \zeta(\al_i) \ga_i^{m-1} \ga_j\, \theta\left( \left[\theta(f_i), \ad(f_i)^{m-1}(f_j) \right] \right) +   [b_i, L_{m-1}].
}
If $\left[\theta(f_i), \ad(f_i)^{m-1}(f_j) \right] \ne 0$ then $w_X(\al_{\tau(i)})-(m-1)\,\al_i-\al_j  \in \Phi \cup \{0\}$.

If $w_X(\al_{\tau(i)})-(m-1)\al_i-\al_j  \in \Phi^+$ we must have $j=\tau(i)$, $X(i) \ne \emptyset$, $m=1$; since $w_X(\al_{\tau(i)})-\al_j   \in Q^+_X$ it follows that $[\theta(f_i),f_j] \in \mfn_X^+$.
The expression for $L_1$ follows from \eqref{recursion:c}; $L_m=0$ with $m>1$ is a consequence of \eqref{k:rels1}.

Now suppose $w_X(\al_{\tau(i)})-(m-1)\,\al_i-\al_j  \in \Phi^- \cup \{0\}$.
It follows that $X(i) = \emptyset$, so $\al_i(\rho^\vee_X) \in \Z$, and $\tau(i) \in \{i,j\}$.
If $\tau(i)=j$ then $\Z \al_i \cap \Phi = \{ \pm \al_i \}$ implies that $m \in \{ 1,2 \}$.
Furthermore, $\theta(f_i)=-e_{j}$ and $a_{ij}=a_{ji}$.
From a successive application of \eqref{recursion:c} we obtain $L_1= \ga_j  h_i -  \ga_i h_{j}$, $L_2 = 2 \left((\ga_j - a_{ij} \ga_i ) f_i  - \ga_i ( \ga_i - a_{ij} \ga_j ) e_{j} \right)$ and $L_m=0$ if $m>2$, as required.

It remains to deal with the case $X(i) = \emptyset$ and $\tau(i)=i$, in which case $\theta(f_i)=-e_i$.
A straightforward computation gives
\[ 
[\theta(f_i),\ad(f_i)^{m-1}(f_j)] = (m-1)(m-2+a_{ij}) \ad(f_i)^{m-2}(f_j).
\]
By virtue of the induction hypothesis, \eqref{recursion:c} simplifies to
\[
L_m = (m-1)(m-1-M_{ij}) \ga_i \left( \ad(b_i)^{m-2}(b_j) - L_{m-2} \right) + [b_i,L_{m-1}],
\]
from which the desired expression follows straightforwardly.
\end{proof}

%%%%%%%%%%%%%%%%%%%%%%%%%%%%%%%%%%%%%%%%%%%%%%%%%%%%%%%%%%%%%%%%%%
% Bibliography
%%%%%%%%%%%%%%%%%%%%%%%%%%%%%%%%%%%%%%%%%%%%%%%%%%%%%%%%%%%%%%%%%%

\enlargethispage{1em}

\end{document}